\documentclass[12pt,letterpaper]{amsart}
\usepackage[margin=1.1in]{geometry}
\usepackage[english]{babel}
\usepackage{amsrefs}

\usepackage{amsmath}
\usepackage{amssymb}
\usepackage{amsthm}
\usepackage{amscd}
\usepackage{bbm}
\usepackage{esvect}
\usepackage{mathrsfs}
\usepackage[inline]{asymptote}
\usepackage{hyperref}
\usepackage{cleveref}
\usepackage{pgf,tikz}
\usepackage{mathrsfs}

\theoremstyle{definition}
\newtheorem{theorem}{Theorem}[section]

\newtheorem{lemma}[theorem]{Lemma}

\newtheorem{proposition}[theorem]{Proposition}

\newtheorem{definition}[theorem]{Definition}
\newtheorem{remark}[theorem]{Remark}
\newtheorem{exmp}[theorem]{Example}

\newcommand{\pref}[1]{(\ref{#1})}

\crefname{subsection}{subsection}{subsections}

\title{Densities for Elliptic Curves over Global Function Fields}
\author{Andrew Yao}

\begin{document}

\begin{abstract}
Let $K$ be a global function field. We obtain a set of formulas for the densities of the Kodaira types and Tamagawa numbers of elliptic curves over a completion of $K$ that is independent of the field's characteristic. Furthermore, for a finite field $F$ and real numbers $s$ and $\epsilon$ such that $s>1$ and $\epsilon>0$, we prove that there exists a global function field $K$ such that the full constant field of $K$ is $F$ and the value of the zeta function of $K$ at $s$ is less than $1+\epsilon$.
\end{abstract}

\maketitle

\section{Introduction}
\label{sec:introduction}

Let $p$ be a prime and $q$ be a power of $p$. Let $K$ be a finite extension of $\mathbb{F}_q(t)$. Define $M_K$ to be the set of places of $K$. Suppose $P\in M_K$. Let $K_P$ be the completion of $K$ at $P$ and $R_P$ be the valuation ring of $K_P$. Suppose $E$ is an elliptic curve over $K$ with equation
\[
E: y^2+a_1xy+a_3y=x^3+a_2x^2+a_4x+a_6
\]
such that $a_1$, $a_2$, $a_3$, $a_4$, and $a_6$ are elements of $K$. $E$ has a long Weierstrass form, and if $a_1=a_2=a_3=0$, $E$ has a short Weierstrass form. We study densities for elliptic curves over $K$ that have a long Weierstrass form. 

As an elliptic curve over $K_P$, $E$ has a Kodaira type, which describes its geometry. Particularly, $E$ has a Tamagawa number $c_P(E):=[E(K_P):E_0(K_P)]$ over $K_P$, where $E_0(K_P)$ is the set of nonsingular points in $E(K_P)$ and $c_P(E)<\infty$. A method to determine the Kodaira type and Tamagawa number of an elliptic curve over $K_P$ is Tate's algorithm (\cite{AdvancedTopicsEC}, \cite{TateAlgorithm}). The description of Tate's algorithm in \cite{AdvancedTopicsEC} is used in this paper to compute local densities. Often, steps from this description of Tate's algorithm are referred to.

The papers \cite{localglobal} and \cite{ProdRationals} discuss densities of Kodaira types and Tamagawa products for elliptic curves over $\mathbb{Q}$. In these papers, the densities at the non-Archimedean places of $\mathbb{Q}$ are considered. In \cite{localglobal} and \cite{ProdRationals}, the densities are for elliptic curves in long and short Weierstrass form, respectively. Moreover, \cite{ProdNF} discusses densities of Kodaira types and Tamagawa products for elliptic curves over number fields in short Weierstrass form. 

Note that some of the methods for computing local densities with Tate's algorithm used in \Cref{sec:locala}, \Cref{sec:localb}, and \Cref{sec:localc} of this paper are similar to methods used in \cite{ProdNF}, \cite{localglobal}, and \cite{ProdRationals}. A goal of this paper is to develop a framework for transforming elliptic curves while applying Tate's theorem and to rigorously compute densities after translations. At each step, we check that the densities that we have computed are correct by returning to the original elliptic curve rather than only considering the transformed elliptic curve. For example, we characterize transformations that convert non-minimal elliptic curves to minimal elliptic curves, see \Cref{prop:unique1} and \Cref{prop:unique2}. We expect that the framework can be applied to other settings as well, for example to compute the corresponding densities over number fields.

Additionally, an important idea of this paper that is not discussed in \cite{localglobal} is the computation of the densities of the Kodaira types $I_N^*$ for individual values of $N\geq 1$. These densities are considered in \cite{ProdNF} and \cite{ProdRationals} for short Weierstrass form. We meticulously analyze these cases in \Cref{subsec:subden1}, \Cref{subsec:subden2}, and \Cref{subsec:subden3}.

Local densities over $K_P$ can be obtained using the Haar measure. Let $N$ be a positive integer. Note that $K_P^N$ as an additive group is locally compact, and because of this, Haar's theorem can be used on $K_P^N$. Particularly, suppose $\mu_P$ is the Haar measure on $K_P^N$ such that $\mu_P(R_P^N)=1$. 

Let $G_P$ be the set of curves $y^2+a_1xy+a_3y=x^3+a_2x^2+a_4x+a_6$ over $K_P$ such that $a_1,a_2,a_3,a_4,a_6\in R_P$. Because the discriminant of an elliptic curve must be nonzero, not all elements of $G_P$ are elliptic curves. Also, note that $G_P$ can be considered to be $R_P^5$. The local densities for $G_P$ are obtained from the Haar measure on $R_P^5$. 

\begin{definition}
\label{def:numiter}
For an elliptic curve $E\in G_P$, let $N_P(E)$ be the number of iterations of Tate's algorithm that are completed when the algorithm is used on $E$.
\end{definition}

Suppose $T$ is the set of Kodaira types. Let $\mathfrak{r}$ be an element of $T$ and $n$ be a positive integer. Define $\delta_K(\mathfrak{r},n;P)$ to be the Haar measure of the set of elliptic curves $E$ over $K_P$ with coefficients in $R_P$ such that $E$ has Kodaira type $\mathfrak{r}$ and the Tamagawa number of $E$ is $n$. For $k\geq 0$, define $\delta_K(\mathfrak{r},n,k;P)$ to be the Haar measure of the set of elliptic curves $E$ over $K_P$ with coefficients in $R_P$ such that $E$ has Kodaira type $\mathfrak{r}$, the Tamagawa number of $E$ is $n$, and $N_P(E)=k$.

The main result that we prove is that given $\mathfrak{r}$ and $n$, $\delta_K(\mathfrak{r}, n; P)$ only depends on $Q_P$, where $Q_P$ is defined in the notation section. In contrast with previous works, the formula for densities does not differ between the cases $p\geq 5$, $p=3$, and $p=2$. \Cref{thm:formula} addresses this result for minimal curves; \Cref{thm:multiterdensity} extends it to non-minimal curves.

\begin{theorem}
\label{thm:formula}
The following are true for all $P\in M_K$:
\begin{itemize}
\item $\delta_K(I_0, 1, 0; P) = \frac{Q_P-1}{Q_P},\,\delta_K(I_1, 1, 0; P) = \frac{(Q_P-1)^2}{Q_P^3}$
\item $\delta_K(I_2, 2, 0; P) = \frac{(Q_P-1)^2}{Q_P^4}$
\item $\delta_K(I_N, N, 0; P) = \delta_K\left(I_N, 2\left\lfloor \frac{N}{2}\right\rfloor - N + 2, 0; P\right) = \frac{(Q_P-1)^2}{2Q_P^{N+2}} \text{ for } N\geq 3$
\item $\delta_K(II, 1, 0; P) = \frac{Q_P-1}{Q_P^3},\,\delta_K(III, 2, 0; P) = \frac{Q_P-1}{Q_P^4}$
\item $\delta_K(IV, 1, 0; P) = \delta_K(IV, 3, 0; P) = \frac{Q_P-1}{2Q_P^5}$
\item $\delta_K(I_0^*, 1, 0; P) = \frac{Q_P^2-1}{3Q_P^7}, \,\delta_K(I_0^*, 2, 0; P) = \frac{Q_P-1}{2Q_P^6},\,\delta_K(I_0^*, 4, 0; P) = \frac{Q_P^2-3Q_P+2}{6Q_P^7}$
\item $\delta_K(I_N^*, 2, 0; P) = \delta_K(I_N^*, 4, 0; P) =\frac{(Q_P-1)^2}{2Q_P^{N+7}} \text{ for } N\geq 1$
\item $\delta_K(II^*, 1, 0; P) = \frac{Q_P-1}{Q_P^{10}},\,\delta_K(III^*, 2, 0; P) = \frac{Q_P-1}{Q_P^9}$
\item $\delta_K(IV^*, 1, 0; P) = \delta_K(IV^*, 3, 0; P) = \frac{Q_P-1}{2Q_P^8}$
\end{itemize}
\end{theorem}

\begin{remark}
In \cite{ProdRationals}, the local densities of $\mathfrak{r}$ and the Tamagawa number $n$ for elliptic curves in short Weierstrass form over $\mathbb{Q}_r$ for primes $r\geq 5$ have the same form as the densities in \Cref{thm:formula}. In \cite{ProdNF}, the local densities of $\mathfrak{r}$ and the Tamagawa number $n$ for elliptic curves in short Weierstrass form over completions of number fields at places that lie above primes $r\geq 5$ also have the same form as these densities.
\end{remark}

\begin{proof}
See Sections \ref{sec:locala}, \ref{sec:localb}, and \ref{sec:localc}.
\end{proof}

In this paper, we often consider the number of iterations that Tate's algorithm completes when the algorithm is used on an elliptic curve over $K_P$. In order to study this question, \Cref{prop:multiterations} is useful. Next, we give an important result of the paper for densities for non-minimal elliptic curves.

\begin{theorem}
\label{thm:multiterdensity}
For a Kodaira type $\mathfrak{r}$, positive integer $n$, and nonnegative integer $k$,
\[
\delta_K(\mathfrak{r},n,k;P)=\frac{1}{Q_P^{10k}}\delta_K(\mathfrak{r}, n, 0; P).
\]
\end{theorem}

We prove \Cref{thm:multiterdensity} by considering the cases $p\geq 5$, $p=3$, and $p=2$. The proof of this result is given in \Cref{subsec:multiterproof}.

Furthermore, we prove the following result in \Cref{subsec:tamagawanum} using results from \Cref{sec:global}. Note that the set $S$ is defined later in this section.

\begin{theorem}
\label{thm:densitytamn}
For $P\in S^C$ and $c\in\mathbb{N}$, let $d_P(c)$ denote the local density of the minimal elliptic curves over $K_P$ with Tamagawa number $c$. For $n\in\mathbb{N}$, the density of the set of minimal curves $E\in W_S$ such that $\prod_{P\in S^C} c_P = n$ is
\[
\sum_{\substack{c_P,\, P \in S^C, \\ \prod_{P\in S^C} c_P = n}} \prod_{P\in S^C} d_P(c_P).
\]
\end{theorem}

\textbf{Organization.} The paper is organized as follows. In \Cref{sec:ec}, we introduce elliptic curves and Tate's algorithm. Next, in \Cref{sec:global}, for a nonempty finite subset $S$ of $M_K$ and a positive integer $N$, we discuss how to obtain global densities for $\mathcal{O}_{K,S}^N$. Afterwards, in Sections \ref{sec:locala}, \ref{sec:localb}, and \ref{sec:localc}, we compute the local densities if the characteristic $p$ of $K$ is at least $5$, equal to $2$, and equal to $3$, respectively. In \Cref{sec:results}, we prove additional results about local and global densities, some of which we have mentioned earlier in this section. In \Cref{sec:construction}, we construct a global function field with zeta function arbitrarily close to one.

\textbf{Notation.} 
Suppose $P$ is a place of $K$. Let $\pi_P$ be a uniformizer of $P$ in $K$.  Let the degree of $P$ be $[R_P/\pi_P R_P:\mathbb{F}_q]$ and let $Q_P=|R_P/\pi_P R_P|$. Also, denote $v_P$ to be the valuation $v_{\pi_P}$ over $K_P$; note that $v_P$ is also a valuation over $K$ because $K\subset K_P$. Moreover for a nonnegative integer $k$, let $L_{P,k}$ be a set of representatives of the cosets of $R_P/\pi_P^kR_P$ such that $0\in L_{P,k}$.

Suppose $S$ is a finite nonempty subset of $M_K$. We let $\mathcal{O}_{K,S}$ be the set of $x\in K$ such that if $P\in S^C=M_K\backslash S$, $v_P(x)\geq 0$. Also, let $W_S$ be the set of curves $y^2+a_1xy+a_3y=x^3+a_2x^2+a_4x+a_6$ such that $a_1,a_2,a_3,a_4,a_6\in\mathcal{O}_{K,S}$.

Suppose $D$ is a divisor of $K$. Define $L(D)$ to be the set of $x\in K$ such that $x=0$ or $x\not=0$ and $(x)+D\geq 0$.

Furthermore, let the zeta function of $K$ be $\zeta_K$. The zeta function is discussed in more detail in \Cref{sec:construction}.

\textbf{Acknowledgments.} The author conducted the research in this paper in the Summer Program in Undergraduate Research at MIT during the summer of 2022. The author would like to thank Hao Peng for providing useful guidance. Also, the author would like to thank Zhiyu Zhang for suggesting the problem. Additionally, the author would like to thank David Jerison and Ankur Moitra for giving advice about the project.

\section{Elliptic Curves}
\label{sec:ec}

Suppose $P$ is a place of $K$. Let $E$ be an elliptic curve over $K_P$. There exist $a_1,a_2,a_3,a_4,\\a_6\in K_P$ such that $E$ has equation
\[
E: y^2+a_1xy+a_3y=x^3+a_2x^2+a_4x+a_6.
\]
Suppose $a_1,a_2,a_3,a_4,a_6\in K_P$ satisfy this condition. Additionally, define 
\begin{align*}
    & b_2(E)=a_1^2+4a_2, b_4(E)=a_1a_3+2a_4, b_6(E)=a_3^2+4a_6, \\
    & b_8(E)=a_1^2a_6+4a_2a_6-a_1a_3a_4+a_2a_3^2-a_4^2.
\end{align*}
he discriminant of $E$ is
\[
\Delta(E)=-b_2(E)^2b_8(E)-8b_4(E)^3-27b_6(E)^2+9b_2(E)b_4(E)b_6(E).
\]

\begin{definition}[\cite{TateAlgorithm}] Elliptic curves $E$ and $F$ over $K_P$ are \textit{isomorphic} if there exists $l,m,n,u\in K_P$ such that $u\not=0$ and the equation for $F$ can be obtained from the equation for $E$ by first replacing $x$ with $u^2x+n$ and $y$ with $u^3y+lu^2x+m$ and then dividing by $u^6$. 
\end{definition}

\begin{definition}[\cite{TateAlgorithm}] 
An elliptic curve $E$ over $K_P$ is $\textit{minimal}$ if the equation for $E$ has coefficients in $R_P$ and if there does not exist an elliptic curve $F$ over $K_P$ such that the equation for $F$ has coefficients in $R_P$, $F$ is isomorphic to $E$, and $v_P(\Delta(F))< v_P(\Delta(E))$.
\end{definition}

The following proposition generalizes Theorem 3.2 of \cite{TateAlgorithm} to non-minimal isomorphic elliptic curves. Note that this proposition is used later in the paper to compute local densities.

\begin{proposition}
\label{translation}
Let $E$ and $F$ be elliptic curves over $K_P$ that have equations with coefficients in $R_P$. Assume that $E$ and $F$ are isomorphic and satisfy $v_P(\Delta(E))=v_P(\Delta(F))$. Then, there exists $l,m,n,u\in R_P$ such that $v_P(u)=0$ and the equation of $F$ can be obtained from the equation of $E$ by first replacing $x$ with $u^2x+n$ and $y$ with $u^3y+lu^2x+m$ and then dividing by $u^6$.
\end{proposition}
\begin{proof}
The proof of Theorem 3.2 of \cite{TateAlgorithm} can be used to prove this proposition.
\end{proof}

\begin{proposition}
\label{prop:multiterations}
     Let $k$ be a nonnegative integer. Suppose $E$ is an elliptic curve over $K_P$ with equation
     \[
     E: y^2+a_1xy+a_3y=x^3+a_2x^2+a_4x+a_6
     \]
     and assume that $a_1,a_2,a_3,a_4,a_6\in R_P$. For $l,m,n\in K_P$, let $E'(l,m,n)$ be the elliptic curve that is $E$ with $x$ replaced by $x+n$ and $y$ replaced by $y+lx+m$. Then, $N_P(E)\geq k$ if and only if there exists $l,m,n\in R_P$ such that if $E'(l,m,n)$ has equation
     \[
     E'(l,m,n):y^2+a_1'xy+a_3'y=x^3+a_2'x^2+a_4'x+a_6',
     \]
     where $a_i'\in \pi_P^{ki}R_P$ for $i\in\{1,2,3,4,6\}$.
\end{proposition}

\begin{proof}
Suppose $l$, $m$, $n$ exist. Let $l,m,n$ satisfy the condition. From Tate's algorithm, we have that $N_P(E)=N_P(E'(l,m,n))\geq k$.

Next, we prove that if $N_P(E)\geq k$, $l$, $m$, and $n$ exist using induction on $k$. The base case $k=0$ is clear. Let $a$ be a nonnegative integer and assume the result is true for $k=a$. We prove the result is true for $k=a+1$. Assume $N_P(E)\geq a+1$. Because $N_P(E)\geq a$, $l,m,n\in R_P$ exist such that if $x$ is replaced with $x+n$ and $y$ is replaced with $y+lx+m$, the resulting curve $E'(l,m,n): y^2+a_1'xy+a_3'y=x^3+a_2'x^2+a_4'x+a_6'$ has $a_i'\equiv 0\pmod{\pi_P^{ia}}$ for $i\in\{1,2,3,4,6\}$. Suppose $l,m,n\in R_P$ satisfy this condition. Suppose that the curve that is obtained after Tate's algorithm is used for $a$ iterations on $E'(l,m,n)$ is
\[
F: y^2+\frac{a_1'}{\pi_P^a}xy+\frac{a_3'}{\pi_P^{3a}}y=x^3+\frac{a_2'}{\pi_P^{2a}}x^2+\frac{a_4'}{\pi_P^{4a}}x+\frac{a_6'}{\pi_P^{6a}}.
\]
We have that $F$ is $E$ with $x$ replaced with $\pi_P^{2a}x+n$ and $y$ replaced with $\pi_P^{3a}y+l\pi_P^{2a}x+m$ divided by $\pi_P^{6a}$.

Because $N_P(E'(l,m,n))=N_P(E)\geq a+1$, $F$ will complete at least one more iteration. During this iteration, suppose $x$ is replaced with $x+n'$ and $y$ is replaced with $y+l'x+m'$. We have that
the resulting elliptic curve
\[
F': y^2+a_1''xy+a_3''y=x^3+a_2''x^2+a_4''x+a_6''
\]
has $a_i''\equiv 0\pmod{\pi_P^i}$ for $i\in\{1,2,3,4,6\}$. Moreover, $F'$ is $E$ with $x$ replaced with 
\[
\pi_P^{2a}x+n+n'\pi_P^{2a}
\]
and $y$ replaced with 
\[
\pi_P^{3a}y+(l+l'\pi_P^a)\pi_P^{2a}x+m+m'\pi_P^{3a}+ln'\pi_P^{2a}\]
divided by $\pi_P^{6a}$. Suppose the equation of
\[
E'(l+l'\pi_P^a, m+m'\pi_P^{3a}+ln'\pi_P^{2a}, n+n'\pi_P^{2a})
\]
is
\[
y^2+a_1'''xy+a_3'''y=x^3+a_2'''x^2+a_4'''+a_6'''.
\]
We have that $a_i'''=\pi_P^{ai}a_i''$ is divisible by $\pi_P^{(a+1)i}$ for $i\in\{1,2,3,4,6\}$. This completes the induction. We are done.
\end{proof}

Note that Tate's algorithm cannot be used on a curve in $G_P$ with discriminant $0$. However, this is not considered in the calculations of local densities later in the paper. Suppose $\mathfrak{r}\in T$, $n$ is a positive integer, and $k$ is a nonnegative integer. The set $U$ of elliptic curves $E\in G_P$ with Kodaira type $\mathfrak{r}$, Tamagawa number $n$, and $M(E)=k$ is an open subset of $G_P$, because if $E\in U$ and multiples of $\pi_P^M$ are added to the coefficients of $E$ for sufficiently positive large integers $M$, the resulting curve will be an element of $U$. Particularly, the set of elliptic curves is an open subset of $G_P$. In the next proposition, we prove that the Haar measure of this set is $1$; note that it follows that the Haar measure of the set of curves in $G_P$ with discriminant $0$ is $0$.

\begin{proposition}
\label{densityec}
The Haar measure of the set of elliptic curves is $1$.
\end{proposition}

\begin{proof}

Let $M$ be a positive integer. For $E: y^2+a_1xy+a_3y=x^3+a_2x^2+a_4x+a_6$, we see that the number of solutions for $a_i$, $i\in\{1,2,3,4,6\}$ modulo $\pi_P^M$ to $\Delta(E)\equiv 0\pmod{\pi_P^M}$ is $O(Q_P^{4M})$. Therefore, the Haar measure of the set of elliptic curves with discriminant equal to $0$ is at most $\frac{O(Q_P^{4M})}{Q_P^{5M}}=O(\frac{1}{Q_P^M})$. The result follows from taking $M\rightarrow\infty$.
\end{proof}

\section{Global Densities}
\label{sec:global}

We discuss results from \cite{globaldensityfunc} and \cite{sqfree} in this section that we use to connect local densities to global densities. Note that \cite{sqfree} considers when $S=\{P_\infty\}$ as well as when $S$ is any finite nonempty subset of $M_K$. We are more interested in the latter case.

\subsection{Setup}
Definitions from \cite{globaldensityfunc} are used in this subsection.

Let $S$ be a finite nonempty subset of $M_K$. Also, suppose $N$ is a positive integer. Let $\text{Div}(S)$ be the set of divisors
\[
\sum_{P\in S} n_PP
\]
such that for $P\in S$, $n_P$ is a nonnegative integer and there exists $P\in S$ such that $n_P>0$. 

Suppose $N$ is a positive integer and suppose $U\subset\mathcal{O}_{K,S}^N$. The upper density of $U$ at $S$ is
\[
\overline{d}_S(U)=\limsup_{D\in\text{Div}(S)}\frac{|U\cap L(D)^N|}{|L(D)|^N}
\]
and the lower density of $U$ at $S$ is
\[
\underline{d}_S(U)=\liminf_{D\in\text{Div}(S)}\frac{|U\cap L(D)^N|}{|L(D)|^N}.
\]
If $\overline{d}_S(U)=\underline{d}_S(U)$, the density $d_S(U)$ of $U$ at $S$ exists and equals $\overline{d}_S(U)=\underline{d}_S(U)$.

\subsection{Results}

\begin{lemma}[{\cite{globaldensityfunc}*{Theorem 2.1}}]
\label{globaldensity1}
For $P\in S^C$, let $U_P\subset K_P^N$ be a measurable set such that $\mu_P(\partial U_P)=0$. For a positive integer $M$, let $V_M$ be the set of $x\in \mathcal{O}_{K,S}^N$ such that $x\in U_P$ for some $P\in S^C$ with degree at least $M$. Suppose
$\lim_{M\rightarrow\infty}\overline{d}_S(V_M)=0$. Let $\mathcal{P}: \mathcal{O}_{K, S}^N\rightarrow 2^{S^C}, \mathcal{P}(a)\triangleq\{P\in S^C: a\in U_P\}$. Then:
\begin{enumerate}
    \item The sum $\sum_{P\in S^C}\mu_P(U_P)$ is convergent.
    \item For $T\subset 2^{S^C}$, $\nu(T):=d_S(\mathcal{P}^{-1}(T))$ exists. Also, $\nu$ defines a measure on $2^{S^C}$.
    \item The measure $\nu$ is concentrated at finite subsets of $S^C$ and for a finite set $T$ of places in $S^C$,
    \[
    \nu(T)=\prod_{P\in T}\mu_P(U_P)\prod_{P\in S^C\backslash T} (1-\mu_P(U_P)).
    \]
\end{enumerate}

\end{lemma}

\begin{lemma}[{\cite{globaldensityfunc}*{Theorem 2.2}}; {\cite{sqfree}*{Proof of Theorem 8.1}}]
\label{globaldensity2}

Let $f$ and $g$ be polynomials in $\mathcal{O}_{K, S}[x_1,\ldots,x_N]$ that are relatively prime. For $M\geq 1$, let $V_M$ be the set of $x\in \mathcal{O}_{K, S}^N$ such that $f(x)\equiv g(x)\equiv 0\pmod{\pi_P}$ for some $P\in S^C$ with degree at least $M$. Then, $\lim_{M\rightarrow\infty}\overline{d}_S(V_M)=0$.
\end{lemma}

\begin{lemma}[{\cite{sqfree}*{Proof of Theorem 8.1}}]
\label{lemma:sqfree}
Let $f\in\mathcal{O}_{K,S}[x_1,\ldots,x_N]$ be square-free as a polynomial in $K[x_1,\ldots,x_N]$. For $M\geq 1$, let $V_M$ be the set of $\mathcal{O}_{K,S}^N$ such that $f(x)\equiv 0 \pmod{\pi_P^2}$ for some $P\in S^C$ with degree at least $M$. Then, $\lim_{M\rightarrow\infty} \overline{d}_S(V_M)=0$.
\end{lemma}

In this paper, we consider global densities for elliptic curves over $K$ with coefficients in $\mathcal{O}_{K,S}$ in long Weierstrass form. We see that $W_S$ can be considered to be $\mathcal{O}_{K,S}^5$, and particularly, the global density definitions from above for $\mathcal{O}_{K,S}^5$ can be used on $W_S$. Similar methods are used in \cite{localglobal} for elliptic curves over $\mathbb{Q}$ with coefficients in $\mathbb{Z}$. Note that an elliptic curve must have a nonzero discriminant, meaning that not all curves in $W_S$ are elliptic curves. However, for $D\in \text{Div}(S)$, the number of curves in $W_S$ with discriminant $0$ that are elements of $L(D)^5$, where $W_S$ is considered to be $\mathcal{O}_{K,S}^5$, is $O(|L(D)|^4)$. Particularly, if proportions over elliptic curves in $W_S$ is considered rather than the proportions over $W_S$, the density is not changed.

\Cref{multitercond} is about the global density of non-minimal elliptic curves. Note that the lemma is used to prove \Cref{multiterglobal}.

\begin{proposition}
\label{multitercond}
For a positive integer $M$, let $V_M$ be the set of elliptic curves $E\in W_S$ such that there exists $P\in S^C$ with degree at least $M$ such that $N_P(E)\geq 1$. Then, $\lim_{M\rightarrow\infty}\overline{d}_S(V_M)=0$.
\end{proposition}

\begin{proof}
We prove this with casework on the characteristic $p$ of $K$. Suppose that $E$ is an elliptic curve in $G_P$ with equation $E:y^2+a_1xy+a_3y=x^3+a_2x^2+a_4x+a_6$ for $a_1,a_2,a_3,a_4,a_6\in R_P$ such that $N_P(E)\geq 1$.

Assume $p\geq 5$. We have that $E$ can be translated to the curve
\[
y^2=x^3+\left(-\frac{b_2(E)^2}{48}+\frac{b_4(E)}{2}\right)x-\frac{b_2(E)^3}{864}-\frac{b_2(E)b_4(E)}{24}+\frac{b_6(E)}{4}.
\]
Because $N_P(E)\geq 1$, using \Cref{prop:multiterations}, $-\frac{b_2(E)^2}{48}+\frac{b_4(E)}{2}\equiv 0\pmod{\pi_P}$ and $-\frac{b_2(E)^3}{864}-\frac{b_2(E)b_4(E)}{24}+\frac{b_6(E)}{4}\equiv 0\pmod{\pi_P}$. Then, \Cref{globaldensity2} with 
\[
f(x_1,x_2,x_3,x_4,x_6)=-\frac{(x_1^2+4x_2)^2}{48}+\frac{x_1x_3+2x_4}{2}
\]
and 
\[
g(x_1,x_2,x_3,x_4,x_6)=-\frac{(x_1^2+4x_2)^3}{864}-\frac{(x_1^2+4x_2)(x_1x_3+2x_4)}{24}+\frac{x_3^2+4x_6}{4}
\]
proves this proposition for $p\geq 5$.

Next, assume $p=3$. We have that $E$ can be translated to the curve
\[
y^2=x^3+\frac{b_2(E)}{4}x^2+\frac{b_4(E)}{2}x+\frac{b_6(E)}{4}
\]
Using \Cref{prop:multiterations}, $\frac{b_2(E)}{4}\equiv 0\pmod{\pi_P}$ from the coefficient of $x^2$. Additionally, $\Delta(E)\equiv 0\pmod{\pi_P}$. Next, \Cref{globaldensity2} with \begin{align*}
f(x_1,x_2,x_3,x_4,x_6)= -(x_1^2+x_2)^2(x_1^2x_6+x_2x_6-x_1x_3x_4+x_2x_3^2-x_4^2)+(x_1x_3+2x_4)^3
\end{align*}
and 
\[
g(x_1,x_2,x_3,x_4,x_6)=x_1^2+x_2
\]
proves this proposition for $p=3$.

Suppose $p=2$. Using \Cref{prop:multiterations}, $a_1\equiv 0\pmod{\pi_P}$ from the coefficient of $xy$. Also, $\Delta(E)\equiv 0\pmod{\pi_P}$. Therefore, \Cref{globaldensity2} with 
\[
f(x_1,x_2,x_3,x_4,x_6)=x_1^4(x_1^2x_6+x_1x_3x_4+x_2x_3^2+x_4^2)+x_3^4+x_1^3x_3^3
\]
and 
\[
g(x_1,x_2,x_3,x_4,x_6)=x_1
\]
proves this proposition for $p=2$.
\end{proof}

\section{Local Densities for \texorpdfstring{$p\geq 5$}{}}
\label{sec:locala}

\subsection{Setup}

Suppose that the characteristic of $K$ is $p\geq 5$. Let $P$ be a place of $K$. We compute the local densities over $K_P$ of Kodaira types $\mathfrak{r}$ and Tamagawa numbers $n$ for elliptic curves in $G_P$. Let $G_P^{(1)}$ be the set of curves
\[
y^2=x^3+a_4x+a_6
\]
over $K_P$ such that $a_4,a_6\in R_P$. Note that $G_P^{(1)}$ can be considered to be $R_P^2$. Define $\varphi: G_P\rightarrow G_P^{(1)}$ as the function such that if $E$ is a curve in $G_P$, $\varphi(E)$ is the curve in $G_P^{(1)}$ with equation
\[
\varphi(E): y^2=x^3+\left(-\frac{b_2(E)^2}{48}+\frac{b_4(E)}{2}\right)x-\frac{b_2(E)^3}{864}-\frac{b_2(E)b_4(E)}{24}+\frac{b_6(E)}{4}.
\]
If $E$ is an elliptic curve, $\varphi(E)$ is an elliptic curve isomorphic to $E$.

\begin{lemma}
\label{invdensity1}
If $U$ is an open subset of $G_P^{(1)}$, $\mu_P(\varphi^{-1}(U))=\mu_P(U)$.
\end{lemma}

\begin{proof}
Let $V$ be the set of $y^2=x^3+a_4'x+a_6'$ with $a_4'\in r_4+\pi_P^{n_4}R_P$ and $a_6'\in r_6+\pi_P^{n_6}R_P$. It suffices to prove that $\mu_P(\varphi^{-1}(V))=\mu_P(V)=\frac{1}{Q^{n_4+n_6}}$ because all open subsets of $G_P^{(1)}$ can be written as a disjoint countable union of sets with the form of $V$. Suppose $E: y^2+a_1xy+a_3y=x^3+a_2x^2+a_4x+a_6\in G_P$. Then, $\varphi(E)\in V$ if and only if
\begin{equation}
\label{eq:coeffx}
-\frac{b_2(E)^2}{48}+\frac{b_4(E)}{2}\in r_4+\pi_P^{n_4}R_P
\end{equation}
and 
\begin{equation}
\label{eq:coeff1}
-\frac{b_2(E)^3}{864}-\frac{b_2(E)b_4(E)}{24}+\frac{b_6(E)}{4}\in r_6+\pi_P^{n_6}R_P.
\end{equation}

Assume that $\varphi(E)\in V$. Let $M=\max(n_4,n_6)$. First, select $a_1$, $a_2$, and $a_3$ modulo $\pi_P^M$. Each has $Q_P^M$ possible residues. Afterwards, $a_4$ will have $Q_P^{M-n_4}$ residues modulo $\pi_P^M$ from (\ref{eq:coeffx}); select the residue for $a_4$. Finally, $a_6$ has $Q_P^{M-n_6}$ residues modulo $\pi_P^M$ from (\ref{eq:coeff1}); select the residue for $a_6$. We see that if each of $a_1,a_2,a_3,a_4,a_6$ are taken modulo $\pi_P^M$, the number of combinations of residues is $Q_P^{5M-n_4-n_6}$. Also, because $a_i$ is modulo $\pi_P^M$ for $i\in\{1,2,3,4,6\}$, each combination of residues has a Haar measure of $\frac{1}{Q_P^{5M}}$. We are done.
\end{proof}

\subsection{Densities after multiple iterations of Tate's algorithm}

Let $k$ be a nonnegative integer. Suppose $S_k$ is the set of elliptic curves $E\in G_P^{(1)}$ such that $N_P(E)\geq k$.

Suppose $E$ is an elliptic curve in $G_P^{(1)}$ with equation $E:y^2=x^3+a_4x+a_6$. Assume $E\in S_k$. Then, using \Cref{prop:multiterations}, $l,m,n\in R_P$ exist such that 
\[
    \left(y+\frac{l}{\pi_P^k}x+\frac{m}{\pi_P^{3k}}\right)^2-\left(x+\frac{n}{\pi_P^{2k}}\right)^3-\frac{a_4}{\pi_P^{4k}}\left(x+\frac{n}{\pi_P^{2k}}\right)-\frac{a_6}{\pi_P^{6k}}\in R_P[x,y].
\]
The coefficient of $xy$ is $\frac{2l}{\pi_P^k}$, giving that $v_P(l)\geq k$, and the coefficient of $y$ is $\frac{2m}{\pi_P^{3k}}$, giving that $v_P(m)\geq 3k$. Also, the coefficient of $x^2$ is $\frac{3n-l^2}{\pi_P^{2k}}$, giving that $v_P(n)\geq 2k$. From this, we have that $v_P(a_4)\geq 4k$ and $v_P(a_6)\geq 6k$. 

Define the function $\phi_k: S_k\rightarrow S_0$, $y^2=x^3+a_4x+a_6\mapsto y^2=x^3+\frac{a_4}{\pi_P^{4k}}x+\frac{a_6}{\pi_P^{6k}}$. Note that $S_k\subset S_0\subset G_P^{(1)}$. From \Cref{densityec} and \Cref{invdensity1}, $\mu_P(S_0)=1$. Next, we show how we can use $\phi_k$ to compute densities for $S_k$.

\begin{lemma}
\label{invdensitymult1}
If $U$ is an open subset of $G_P^{(1)}$, $\mu_P(\phi_k^{-1}(U))=\frac{1}{Q_P^{10k}}\mu_P(U)$.
\end{lemma}

\begin{proof}
Suppose $r_4, r_6\in R_P$. Also, suppose $n_4$ and $n_6$ are nonnegative integers. Let $V$ be the set of elliptic curves $y^2=x^3+a_4'x+a_6'$ with $a_4'\in r_4+\pi_P^{n_4}R_P$ and $a_6'\in r_6+\pi_P^{n_6}R_P$. Because $\mu_P(S_0)=1$, $\mu_P(V)=\frac{1}{Q_P^{n_4+n_6}}$. To prove the lemma, it suffices to prove that 
\[
\mu_P(\phi_k^{-1}(V))=\frac{1}{Q_P^{10k}}\mu_P(V)=\frac{1}{Q_P^{n_4+n_6+10k}}.
\]

Suppose $E:y^2=x^3+a_4x+a_6\in G_P^{(1)}$ is an elliptic curve. We prove that $E\in S_k$ and $\phi_k(E)\in V$ if and only if $\frac{a_4}{\pi_P^{4k}}\in r_4+\pi_P^{n_4}R_P$ and $\frac{a_6}{\pi_P^{6k}}\in r_6+\pi_P^{n_6}R_P$. If $\phi_k(E)\in V$, then $\frac{a_4}{\pi_P^{4k}}\in r_4+\pi_P^{n_4}R_P$ and $\frac{a_6}{\pi_P^{6k}}\in r_6+\pi_P^{n_6}R_P$. Assume that $\frac{a_4}{\pi_P^{4k}}\in r_4+\pi_P^{n_4}R_P$ and $\frac{a_6}{\pi_P^{6k}}\in r_6+\pi_P^{n_6}R_P$. From Tate's algorithm, we have that $E\in S_k$. Then, it is true that $\phi_k(E)\in V$.

Assume that $E\in S_k$ and $\phi_k(E)\in V$. This is true if and only if $a_4\in \pi_P^{4k}r_4+\pi_P^{n_4+4k}R$ and $a_6\in \pi_P^{6k}r_6+\pi_P^{n_6+6k}R$. Moreover, because $\mu_P(S_0)=1$, the density of curves $y^2=x^3+a_4x+a_6$ with discriminant $0$ such that $a_4\in \pi_P^{4k}r_4+\pi_P^{n_4+4k}$ and $a_6\in \pi_P^{6k}r_6+\pi_P^{n_6+6k}$ is $0$. Because of this, $\mu_P(\phi_k^{-1}(V))=\frac{1}{Q_P^{n_4+n_6+10k}}$, completing the proof.
\end{proof}

\subsection{Density calculations}

Note that the density of a set of curves in $G_P^{(1)}$ is the Haar measure of the set. In this subsection, we compute the density of the set of minimal elliptic curves with a given Kodaira type and Tamagawa number over $G_P^{(1)}$. This can be extended to non-minimal elliptic curves using \Cref{thm:multiterdensity}. Moreover, in this subsection, we use the fact that the set of curves in $G_P^{(1)}$ that have a discriminant equal to $0$ has a Haar measure of $0$.

Suppose the discriminant is not divisible by $\pi_P$. We compute the density for this set by considering $a_4$ and $a_6$ modulo $\pi_P$. Suppose $a_4\in r_4+\pi_PR_P$ and $a_6\in r_6+\pi_PR_P$. We find the number of pairs $(r_4, r_6)$ in $L_{P,1}^2$ such that $\left(\frac{r_4}{3}\right)^3+\left(\frac{r_6}{2}\right)^2\equiv 0\pmod{\pi_P}$. If $r_4=0$, $r_6$ has $1$ choice, and if $-\frac{r_4}{3}$ is a square modulo $\pi_P$, $r_6$ has $2$ choices. Otherwise, $r_6$ has $0$ choices. We see that the number of pairs $(r_4, r_6)$ is $Q_P$. Therefore, where each pair $(r_4, r_6)$ has a density of $\frac{1}{Q_P^2}$, the density of the discriminant not being divisible by $\pi_P$ is $\frac{Q_P-1}{Q_P}$. For this case, Tate's algorithm ends in step 1 and we get that $\delta_K(I_0,1,0; P)=\frac{Q_P-1}{Q_P}$.

Next, assume that the discriminant is divisible by $\pi_P$. Furthermore, assume that $a_4,a_6\not\equiv 0\pmod{\pi_P}$. Because there are $Q_P-1$ pairs $(r_4, r_6)$ in $L_{P,1}^2$ for this case, the total density is $\frac{Q_P-1}{Q_P^2}$. Let $\alpha$ be the element of $L_{P,1}$ such that $a_4\equiv -3\alpha^2\pmod{\pi_P}$ and $a_6\equiv 2\alpha^3\pmod{\pi_P}$. The singular point is $(\alpha, 0)$ and in step 2, $x$ is replaced with $x+n$ where $n=\alpha$. Because $\alpha\not\equiv 0\pmod{\pi_P}$, Tate's algorithm ends in step 2. The quadratic considered in step 2 is $T^2-3\alpha$. We see that for $\frac{Q_P-1}{2}$ values of $\alpha$, this quadratic has roots in $R_P/\pi_P R_P$ and $c=v_P(\Delta(E))$, where $c$ denotes the Tamagawa number of $E$. Otherwise, $c=1$ if $v_P(\Delta(E))$ is odd and $c=2$ if $v_P(\Delta(E))$ is even.

Let $N$ be a positive integer. Suppose $a_4\in r_4+\pi_P^NR_P$ and $a_6\in r_6+\pi_P^NR_P$. We find the number of pairs $(r_4, r_6)$ in $L_{P,1}^2$ such that $\left(\frac{r_4}{3}\right)^3+\left(\frac{r_6}{2}\right)^2\equiv 0\pmod{\pi_P^N}$ and $r_4, r_6\not=0$. Because there are $\frac{Q_P^N-Q_P^{N-1}}{2}$ nonzero residues that are squares modulo $\pi_P^N$, we have that the number of pairs $(r_4, r_6)$ is $Q_P^N-Q_P^{N-1}$. Therefore, the density of $v_P(\Delta(E))\geq N$ for $a_4, a_6\not\equiv 0\pmod{\pi_P}$ is $\frac{Q_P-1}{Q_P^{N+1}}$. 

Suppose $N$ is a positive integer. The density of $v_P(\Delta(E))=N$ is $\frac{Q_P-1}{Q_P^{N+1}}-\frac{Q_P-1}{Q_P^{N+2}}=\frac{(Q_P-1)^2}{Q_P^{N+2}}$. We therefore have that $\delta_K(I_1, 1, 0; P)=\frac{(Q_P-1)^2}{Q_P^3}$, $\delta_K(I_2, 2, 0; P)=\frac{(Q_P-1)^2}{Q_P^4}$, and 
\[
\delta_K(I_N, N, 0;P)=\delta_K\left(I_N,2\left\lfloor\frac{N}{2}\right\rfloor-N+2,0;P\right)=\frac{(Q_P-1)^2}{2Q_P^{N+2}}
\]
for $N\geq 3$. 

If $v_P(a_4), v_P(a_6)\geq 1$, the singular point modulo $\pi_P$ from step 2 of Tate's algorithm is $(0,0)$. The total density for this case is $\frac{1}{Q_P^2}$. If $v_P(a_6)=1$, the algorithm ends in step 3. For this case, we get that $\delta_K(II, 1, 0;P)=\frac{Q_P-1}{Q_P^3}$.

Assume that $v_P(a_6)\geq 2$. The total density for this case is $\frac{1}{Q_P^3}$. If $v_P(a_4)=1$, the algorithm ends in step 4 and we get that $\delta_K(III, 2,0;P)=\frac{Q_P-1}{Q_P^4}$.

Next, suppose $v_P(a_4)\geq 2$. The total density for this case is $\frac{1}{Q_P^4}$. If $v_P(a_6)=2$, the algorithm ends in step 5. From this, we have that $\delta_K(IV,1,0;P)=\delta_K(IV,3,0;P)=\frac{Q_P-1}{2Q_P^5}$.

Suppose $v_P(a_6)\geq 3$. The total density for this case is $\frac{1}{Q_P^5}$. In step 6, the polynomial $P(T)\in (R_P/\pi_PR_P)[T]$ has coefficient of $T^2$ equal to $0$. From adding multiples of $\pi_P^2$ to $a_4$, the choices for the coefficient of $T$ are $L_{P,1}$. Also, from adding multiples of $\pi_P^3$ to $a_6$, the choices for the constant term are $L_{P,1}$. Then, we have that each polynomial $P(T)\in (R_P/\pi_PR_P)[T]$ with coefficient of $T^2$ equal to $0$ corresponds to a density of $\frac{1}{Q_P^7}$ in $G_P^{(1)}$.

Assume $P(T)$ has distinct roots in $\overline{R_P/\pi_P R_P}$. The total number of choices for $P(T)$ is $Q_P^2-Q_P$; therefore, the total density for this case is $\frac{Q_P-1}{Q_P^6}$. We have that Tate's algorithm ends in step 6 here. The number of $P(T)$ with $0$, $1$, and $3$ roots in $R_P/\pi_P R_P$ is $\frac{Q_P^2-1}{3}$, $\frac{Q_P^2-Q_P}{2}$, and $\frac{Q_P^2-3Q_P+2}{6}$, respectively. With this, $\delta_K(I_0^*, 1, 0; P)=\frac{Q_P^2-1}{3Q_P^7}$, $\delta_K(I_0^*, 2, 0; P)=\frac{Q_P-1}{2Q_P^6}$, and $\delta_K(I_0^*, 4, 0; P)=\frac{Q_P^2-3Q_P+2}{6Q_P^7}$.

Next, assume that $P(T)$ has a double root and a simple root in $\overline{R_P/\pi_P R_P}$. Then, Tate's algorithm enters the subprocedure in step 7. For this case, the total number of $P(T)$ is $Q_P-1$ and the total density is therefore $\frac{Q_P-1}{Q_P^7}$. In \Cref{subsec:subden1}, we compute that $\delta_K(I_N^*,2,0;P)=\delta_K(I_N^*,4,0;P)=\frac{(Q_P-1)^2}{2Q_P^{N+7}}$ for all positive integers $N$.

Assume $P(T)$ has a triple root in $\overline{R_P/\pi_P R_P}$. For this case, the total number of $P(T)$ is $1$ and the total density is therefore $\frac{1}{Q_P^7}$. Because the coefficient of $T^2$ in $P(T)$ is $0$, the triple root is $0$. If $v_P(a_6)=4$, the algorithm ends in step 8. For this case, $\delta_K(IV^*, 1, 0; P)=\delta_K(IV^*, 3, 0; P) = \frac{Q_P-1}{2Q_P^8}$.

Next, assume that $v_P(a_6)\geq 5$. The total density for this case is $\frac{1}{Q_P^8}$. If $v_P(a_4)=3$, the algorithm ends in step 9. We then have that $\delta_K(III^*, 2, 0; P)=\frac{Q_P-1}{Q_P^9}$.

Suppose $v_P(a_4)\geq 4$. The total density for this case is $\frac{1}{Q_P^9}$. If $v_P(a_6)=5$, the algorithm ends in step 10. Therefore, $\delta_K(II^*,1,0; P)=\frac{Q_P-1}{Q_P^{10}}$.

With density $\frac{1}{Q_P^{10}}$, we have that $v_P(a_4)\geq 4$ and $v_P(a_6)\geq 6$, meaning that the curve is not minimal. That is, the curve will complete iteration $1$ and continue iteration $2$. Note that the density of non-minimal curves calculated from the algorithm matches \Cref{invdensitymult1}.

\subsection{Subprocedure density calculations}
\label{subsec:subden1}

Next, we study the densities for the subprocedure in step 7 of Tate's algorithm. We compute the subprocedure densities by studying the translation of $x$ in Tate's algorithm. In the step 7 subprocedure, because the coefficient of $y$ is initially $0$, there will be no translations of $y$.

Let $X$ be the set of elliptic curves $E\in G_P^{(1)}$ such that $N_P(E)=0$ and Tate's algorithm enters the step 7 subprocedure when used on $E$. For $E\in X$, let $L(E)$ be the number of iterations of the step 7 subprocedure that are completed when Tate's algorithm is used on $E$. For a nonnegative integer $N$, let $X_N$ be the set of $E\in X$ such that $L(E)\geq N$.

Suppose $N$ is an even nonnegative integer. Iteration $N$ of the step 7 subprocedure is completed if and only if $n\in R_P$ exists such that $v_P(n)=1$, $v_P(a_4+3n^2)\geq\frac{N+6}{2}$, and $v_P(n^3+3na_4+a_6)\geq N+4$. Suppose $n=n_1$ satisfies this condition. Suppose $n=n_2$ also satisfies this condition. We then have that $n_1^2\equiv n_2^2\pmod{\pi_P^{\frac{N+6}{2}}}$. This gives that $n_1$ is equivalent to $n_2$ or $-n_2$ modulo $\pi_P^{\frac{N+4}{2}}$. However, because $n_1^3+n_1a_4\equiv n_2^3+n_2a_4\pmod{\pi_P^{N+4}}$, we have that $v_P(n_1-n_2)\geq\frac{N+4}{2}$. Moreover, if $v_P(n_1-n_2)\geq\frac{N+4}{2}$, $n=n_2$ also satisfies the condition.

Next, suppose $N$ is an odd nonnegative integer. Iteration $N$ of the subprocedure is completed if and only if $n\in R_P$ exists such that $v_P(n)=1$, $v_P(a_4+3n_1^2)\geq\frac{N+5}{2}$, and $v_P(n^3+na_4+a_6)\geq N+4$. Similarly, we have that if $n=n_1$ satisfies the condition, $n=n_2$ satisfies the condition if and only if $v_P(n_1-n_2)\geq \frac{N+3}{2}$.

Suppose $N$ is a nonnegative integer. Suppose $n$ is an element of $L_{P,\left\lfloor\frac{N+4}{2}\right\rfloor}$ such that $v_P(n)=1$. Let $Y_{n,N}$ be the set of curves $x^3+3nx^2+a_4'x+a_6'$ such that $v_P(a_4')\geq\left\lfloor\frac{N+6}{2}\right\rfloor$ and $v_P(a_6')\geq N+4$. Note that $Y_{n,N}$ can be considered to be an open subset of $R_P^2$.

For $E\in X_N$, let $n_N(E)$ be the unique value of $n\in L_{P, \left\lfloor\frac{N+4}{2}\right\rfloor}$ such that $v_P(n)=1$, $v_P(a_4+3n^2)\geq\left\lfloor\frac{N+6}{2}\right\rfloor$, and $v_P(n^3+na_4+a_6)\geq N+4$. Let $\theta_N$ be the function such that if $E: y^2=x^3+a_4x+a_6$ is an element of $X_N$, 
\begin{align*}
\theta_N(E): y^2 & =(x+n_N(E))^3+a_4(x+n_N(E))+a_6 \\
& =x^3+3n_N(E)x^2+(a_4+3n_N(E)^2)x+n_N(E)a_4+a_6+n_N(E)^3.
\end{align*}

\begin{lemma}
\label{invdensitysub1} If $U$ is an open subset of $Y_{n,N}$, $\mu_P(\theta_N^{-1}(U))=\mu_P(U)$.
\end{lemma}

\begin{proof}
Suppose $r_4, r_6\in R_P$. Also, suppose $n_4$ and $n_6$ are nonnegative integers. Assume that $v_P(r_4), n_4\geq\lfloor\frac{N+4}{2}\rfloor$ and $v_P(r_6), n_6\geq N+4$. Let $V\subset Y_{n,N}$ be the set of $E':y^2=x^3+3nx^2+a_4'x+a_6'$ such that $a_4'\in r_4+\pi_P^{n_4}R_P$ and $a_6'\in r_6+\pi_P^{n_6}R_P$. It suffices to prove that $\mu_P(\theta_N^{-1}(V))=\mu_P(V)$. Suppose $E: y^2=x^3+a_4x+a_6$ is an elliptic curve. 

We prove that that $E\in X_N$ and $\theta_N(E)\in V$ if and only if
\[
a_4+3n^2\in r_4+\pi_P^{n_4}R_P, na_4+a_6+n^3\in r_6+\pi_P^{n_6}R_P.
\]
Assume that $E\in X_N$ and $\theta_N(E)\in V$. Because $\theta_N(E)\in V$, we have that $n_N(E)=n$. Therefore, $a_4+3n^2\in r_4+\pi_P^{n_4}R_P$ and $na_4+a_6+n^3\in r_6+\pi_P^{n_6}R_P$. Next, assume that $a_4+3n^2\in r_4+\pi_P^{n_4}$ and $na_4+a_6+n^3\in r_6+\pi_P^{n_6}R_P$. Because $v_P(a_4+3n^2)\geq \left\lfloor\frac{N+6}{2}\right\rfloor$ and $v_P(na_4+a_6+n^3)\geq N+4$, $E\in X_N$. We then have that $\theta_N(E)\in V$.

Let $M=\max(n_4, n_6)$. Modulo $\pi_P^M$, there are $Q_P^{M-n_4}$ choices for the residue of $a_4$. After choosing $a_4$ modulo $\pi_P^M$, there are $Q_P^{M-n_6}$ choices for the residue of $a_6$ modulo $\pi_P^M$. Each of these combinations of residues modulo $\pi_P^M$ for $a_4$ and $a_6$ has a density of $\frac{1}{Q_P^{2M}}$ in $G_P^{(1)}$. The Haar measure of the $Q_P^{2M-n_4-n_6}$ combinations is $\frac{1}{Q_P^{n_4+n_6}}$. Because the set of curves in $G_P^{(1)}$ with discriminant $0$ has a Haar measure of $0$,
\[
\mu_P(\theta_N^{-1}(V)) = \frac{1}{Q_P^{n_4+n_6}} = \mu_P(V).
\]
This finishes the proof.
\end{proof}

Let $N$ be a positive integer. We compute the density of $I_N^*$. Let $n$ be an element of $L_{P, \left\lfloor\frac{N+3}{2}\right\rfloor}$ such that $v_P(n)=1$. We have that the Haar measure of the set of $E\in Y_{n,N-1}$ that do not complete iteration $N$ is $\frac{Q_P-1}{Q_P^{\left\lfloor\frac{N+5}{2}\right\rfloor+N+4}}$. With \Cref{invdensitysub1}, because there are $(Q_P-1)Q_P^{\left\lfloor\frac{N-1}{2}\right\rfloor}$ values of $n$, the density of $I_N^*$ is $\frac{(Q_P-1)^2}{Q_P^{N+7}}$. From adding multiples of $\pi_P^{N+3}$ to $a_6$, $c=2$ and $c=4$ have equal density. Therefore,
\[
\delta_K(I_N^*, 2, 0; P) = \delta_K(I_N^*, 4, 0; P) = \frac{(Q_P-1)^2}{2Q_P^{N+7}}.
\]

\section{Local Densities for \texorpdfstring{$p=3$}{}}
\label{sec:localb}

\subsection{Setup}

Suppose that the characteristic of $K$ is $p=3$. Let $P$ be a place of $K$ and $G_P^{(2)}$ be the set of curves
\[
y^2=x^3+a_2x^2+a_4x+a_6
\]
over $K_P$ such that $a_2, a_4, a_6\in R_P$. Note that $G_P^{(2)}$ can be considered to be $R_P^3$. Define $\varphi: G_P\rightarrow G_P^{(2)}$ as the function such that if $E$ is a curve in $G_P$, $\varphi(E)$ is the curve in $G_P^{(2)}$ with equation
\[
y^2=x^3+\frac{b_2(E)}{4}x^2+\frac{b_4(E)}{2}x+\frac{b_6(E)}{4}.
\]
Note that if $E$ is an elliptic curve, $E$ and $\varphi(E)$ are isomorphic.

\begin{lemma}
\label{invdensity2} If $U$ is an open subset of $G_P^{(2)}$, $\mu_P(\varphi^{-1}(U))=\mu_P(U)$.
\end{lemma}
\begin{proof}
This can be proved using a method similar to the proof of \Cref{invdensity1}.
\end{proof}

\subsection{Densities after multiple iterations of Tate's algorithm}

Let $k$ be a nonnegative integer. Suppose $S_k$ is the set of elliptic curves $E\in G_P^{(2)}$ such that $N_P(E)\geq k$.

For an elliptic curve $E\in G_P^{(2)}$ with equation $E: y^2=x^3+a_2x^2+a_4x+a_6$, let $A_k(E)$ be the set of $n\in R_P$ such that
\[
y^2=x^3+\frac{a_2}{\pi_P^{2k}}x^2+\frac{2na_2+a_4}{\pi_P^{4k}}x+\frac{n^2a_2+na_4+a_6+n^3}{\pi_P^{6k}}
\]
has coefficients in $R_P$. The next proposition is useful for computing local densities for multiple iterations.

\begin{theorem}
\label{prop:unique1}
Let $E$ be an elliptic curve in $G_P^{(2)}$. Then, $E\in S_k$ if and only if a unique element $n\in L_{P,k}$ exists such that $n\in A_k(E)$.
\end{theorem}

\begin{proof}
Assume a unique element $n\in L_{P,k}$ exists such that $n\in A_k(E)$. Then, $A_k(E)$ is nonempty, and using \Cref{prop:multiterations}, $E\in S_k$.

Next, assume $E\in S_k$. Suppose $E\in S_k$ has equation $E:y^2=x^3+a_2x^2+a_4x+a_6$. From \Cref{prop:multiterations}, $l,m,n\in R_P$ exist such that
\[
    \left(y+\frac{l}{\pi_P^k}x+\frac{m}{\pi_P^{3k}}\right)^2=\left(x+\frac{n}{\pi_P^{2k}}\right)^3+\frac{a_2}{\pi_P^{2k}}\left(x+\frac{n}{\pi_P^{2k}}\right)^2+\frac{a_4}{\pi_P^{4k}}\left(x+\frac{n}{\pi_P^{2k}}\right)+\frac{a_6}{\pi_P^{6k}}
\]
has coefficients in $R_P$. From the coefficient of $xy$, $v_P(l)\geq k$, and from the coefficient of $y$, $v_P(m)\geq 3k$. Therefore, we have that
\[
y^2=\left(x+\frac{n}{\pi_P^{2k}}\right)^3+\frac{a_2}{\pi_P^{2k}}\left(x+\frac{n}{\pi_P^{2k}}\right)^2+\frac{a_4}{\pi_P^{4k}}\left(x+\frac{n}{\pi_P^{2k}}\right)+\frac{a_6}{\pi_P^{6k}}
\]
has coefficients in $R_P$. Note that $v_P(a_2)\geq 2k$ also. We therefore have that $A_k(E)$ is nonempty. 

Suppose $n\in A_k(E)$. From replacing $x$ with $x+n'$ for $n'\in R_P$, we have that $n+n'\pi_P^{2k}\in A_k(E)$. Therefore, $n\in L_{P,k}$ exists such that $n\in A_k(E)$.

Next, we prove uniqueness. Assume $n_1,n_2\in A_k(E)\cap L_{P,k}$. Let
\[
F:y^2=x^3+\frac{a_2}{\pi_P^{2k}}x^2+\frac{a_4}{\pi_P^{4k}}x+\frac{a_6}{\pi_P^{6k}}.
\]
For $1\leq i\leq 2$, let $F_i$ be $F$ with $x$ replaced by $x+\frac{n_i}{\pi_P^{2k}}$. Note that $F_1,F_2\in G_P^{(2)}$. 

From the coefficients of $x$ in $F_1$ and $F_2$, 
\[
2n_1a_2+a_4\equiv 2n_2a_2+a_4\equiv 0\pmod{\pi_P^{4k}}.
\]
Also, from the constant terms of $F_1$ and $F_2$,
\[
n_1^2a_2+n_1a_4+n_1^3 \equiv n_2^2a_2+n_2a_4+n_2^3\pmod{\pi_P^{6k}}.
\]
For the sake of contradiction, assume that $v_P(n_1-n_2)<2k$. Let $a=v_P(n_1-n_2)$. Note that 
\[
v_P(n_1^3-n_2^3)=v_P((n_1-n_2)^3)=3a.
\]
We have that
\[
n_1^2a_2+n_1a_4-n_2^2a_2-n_2a_4=(n_1-n_2)(n_1a_2+n_2a_2+a_4).
\]
Because $a_4\equiv n_1a_2\equiv n_2a_2\pmod{\pi_P^{4k}}$, 
\[
n_1a_2+n_2a_2+a_4\equiv 3a_4\equiv 0\pmod{\pi_P^{4k}}.
\]
From this,
\begin{align*}
v_P(n_1^2a_2+n_1a_4-n_2^2a_2-n_2a_4) & =
v_P((n_1-n_2)(n_1a_2+n_2a_2+a_4))
\geq a+4k>3a.
\end{align*}
Since $v_P(n_1^3-n_2^3)=3a$,
\[
v_P(n_1^2a_2+n_1a_4+n_1^3-n_2^2a_2-n_2a_4-n_2^3)=3a<6k, 
\]
which is a contradiction. Therefore, $v_P(n_1-n_2)\geq 2k$ and $n_1=n_2$.
\end{proof}

Using \Cref{prop:unique1}, for $E\in S_k$, let $n(E)$ be the unique $n\in L_{P,2k}$ such that $n\in A_k(E)$. Define $\phi_k: S_k\rightarrow S_0$ to be the function such that if $E\in S_k$ has equation $E:y^2=x^3+a_2x^2+a_4x+a_6$, $\phi_k(E)\in S_0$ has equation
\[
\phi_k(E):y^2=x^3+\frac{a_2}{\pi_P^{2k}}x^2+\frac{2n(E)a_2+a_4}{\pi_P^{4k}}x+\frac{n(E)^2a_2+n(E)a_4+a_6+n(E)^3}{\pi_P^{6k}}.
\]
Note that $S_k\subset S_0\subset G_P^{(2)}$. Also, using \Cref{densityec} and \Cref{invdensity2}, $\mu_P(S_0)=1$. For $n\in L_{P, 2k}$, suppose $S_{k,n}$ is the set of $E\in S_k$ such that $n(E)=n$ and let $\phi_{k,n}$ be $\phi_k$ restricted to $S_{k,n}$.

\begin{lemma}
\label{invdensitymult2a}
Suppose $n\in L_{P,k}$. If $U$ is an open subset of $G_P^{(2)}$, $\mu_P(\phi_{k,n}^{-1}(U))=\frac{1}{Q_P^{12k}}\mu_P(U)$.
\end{lemma}

\begin{proof}
Suppose $r_2, r_4, r_6\in R_P$. Also, suppose $n_2$, $n_4$, and $n_6$ are nonnegative integers. Let $V$ be the set of $y^2=x^3+a_2'x^2+a_4'x+a_6'$ such that $a_2'\in r_2+\pi_P^{n_2}R_P$, $a_4'\in r_4+ \pi_P^{n_4}R_P$, and $a_6'\in r_6+\pi_P^{n_6}R_P$. Suppose $E:y^2=x^3+a_2x^2+a_4x+a_6\in G_P^{(2)}$. 
Then, $E\in S_{k,n}$ and $\phi_{k,n}(E)\in V$ if and only if
\begin{align*}
& \frac{a_2}{\pi_P^{2k}}\in r_2+\pi_P^{n_2}R_P,\, \frac{2na_2+a_4}{\pi_P^{4k}}\in r_4+\pi_P^{n_4}R_P,\,\frac{n^2a_2+na_4+a_6+n^3}{\pi_P^{6k}}\in r_6+\pi_P^{n_6} R_P.
\end{align*}

Assume that $E\in S_{k,n}$ and $\phi_{k,n}(E)\in V$. Let $M=\max(n_2+2k, n_4+4k, n_6+6k)$. There are $Q_P^{M-n_2-2k}$ ways to pick $a_2$ modulo $\pi_P^M$. Afterwards, $a_4$ will have $Q_P^{M-n_4-4k}$ choices for its residue modulo $\pi_P^M$. Select the residue for $a_4$. Next, $a_6$ has $Q_P^{M-n_6-6k}$ choices for its residue modulo $\pi_P^M$. Select the residue for $a_6$. The number of combinations of residues is $Q_P^{3M-n_2-n_4-n_6-12k}$ and each combination of residues has a Haar measure of $Q_P^{-3M}$. Also, because $\mu_P(S_0)=1$, the set of curves with discriminant $0$ counted in these combinations of residues has a Haar measure $0$. Therefore, $\mu_P(\phi_{k,n}^{-1}(V))=\frac{1}{Q_P^{n_2+n_4+n_6+12k}}$. With this, $\mu_P(\phi_{k,n}^{-1}(U))=\frac{1}{Q_P^{12k}}\mu_P(U)$ for all open subsets $U$ of $G_P^{(2)}$.
\end{proof}

\begin{lemma}
\label{invdensitymult2b}
If $U$ is an open subset of $G_P^{(2)}$, $\mu_P(\phi_k^{-1}(U))=\frac{1}{Q_P^{10k}}\mu_P(U)$.
\end{lemma}

\begin{proof}
Let $U$ be an open subset of $G_P^{(2)}$ We have that $\phi_k^{-1}(U)=\bigsqcup_{n\in L_{P, 2k}} \phi_{k,n}^{-1}(U)$. Using \Cref{invdensitymult2a},
\[
\mu_P(\phi_k^{-1}(U))=\sum_{n\in L_{P,2k}}\mu_P(\phi_{k,n}^{-1}(U))=\sum_{n\in L_{P,2k}} \frac{1}{Q_P^{12k}}\mu_P(U)= \frac{1}{Q_P^{10k}}\mu_P(U),
\]
completing the proof.
\end{proof}

\subsection{Density calculations for \texorpdfstring{$v_P(a_2)=0$}{}}

Suppose $v_P(a_2)=0$. The density for this case over $G_P^{(2)}$ is $\frac{Q_P-1}{Q_P}$. The discriminant is $-a_2^3a_6+a_2^2a_4^2-a_4^3$.

From adding multiples of $\pi_P$ to $a_6$, the set of curves with discriminant not divisible by $\pi_P$ has density $\frac{(Q_P-1)^2}{Q_P^2}$. Then, we add $\frac{(Q_P-1)^2}{Q_P^2}$ to $\delta_K(I_0,1,0;P)$.

Assume the discriminant is divisible by $\pi_P$. The algorithm ends in step 2. Because $v_P(a_2)=0$, the coefficient of $a_6$ in the discriminant is not divisible by $\pi_P$. Then, we see that for $N\geq 0$, the density over $G_P^{(2)}$ of curves such that $v_P(a_2)=0$ and $v_P(\Delta(E))=N$ is $\frac{(Q_P-1)^2}{Q_P^{N+2}}$. If $a_2\equiv r_2\pmod{\pi_P}$ for $r_2\in L_{P,1}$ such that $r_2\not=0$, $T^2+a_2$ is irreducible over $R_P/\pi_P R_P$ for $\frac{Q_P-1}{2}$ values of $r_2$. Using step 2 of Tate's algorithm, we have that $\delta_K(I_1, 1, 0; P)=\frac{(Q_P-1)^2}{Q_P^3}$, $\delta_K(I_2, 2, 0; P)=\frac{(Q_P-1)^2}{Q_P^4}$, and 
\[
\delta_K(I_N, N, 0;P)=\delta_K\left(I_N,2\left\lfloor\frac{N}{2}\right\rfloor-N+2,0;P\right)=\frac{(Q_P-1)^2}{2Q_P^{N+2}}
\]
for $N\geq 3$. 

\subsection{Density calculations for \texorpdfstring{$v_P(a_2)\geq 1$}{}}

Next, suppose $v_P(a_2)\geq 1$. The density for this case is $\frac{1}{Q_P}$ and modulo $\pi_P$, the discriminant is $-a_4^3$.

Assume the discriminant is not divisible by $\pi_P$. This occurs if and only if $a_4$ is not divisible by $\pi_P$ and the density for this case is $\frac{Q_P-1}{Q_P^2}$. Adding this density to $\delta_K(I_0,1,0;P)$ gives that $\delta_K(I_0,1,0;P)=\frac{Q_P-1}{Q_P}$.

Next, assume the discriminant is divisible by $\pi_P$. The total density for the following cases will be $\frac{1}{Q_P^2}$. Suppose $\alpha_1$ is an element of $L_{P, 1}$ such that $a_6+\alpha_1^3\equiv 0\pmod{\pi_P}$. A singular point is $(\alpha_1,0)$. We have that $x$ is replaced with $x+n$ where $n=\alpha_1$. The resulting curve has equation
\[
y^2=(x+n)^3+a_2(x+n)^2+a_4(x+n)+a_6.
\]
We have that $n^2a_2+n a_4+a_6+n^3$ is not divisible by $\pi_P^2$ with density $\frac{Q_P-1}{Q_P^3}$ by adding multiples of $\pi_P$ to $a_6$. Afterwards, we obtain that $\delta_K(II,1,0;P)=\frac{Q_P-1}{Q_P^3}$.

Assume $n^2a_2+n a_4+a_6+n^3$ is divisible by $\pi_P^2$. The total density for this case is $\frac{1}{Q_P^3}$. The density of $v_P(2n a_2+a_4)=1$ is $\frac{Q_P-1}{Q_P^4}$ from replacing $a_4$ with $a_4+\pi_Pd$ and $a_6$ with $a_6-\alpha_1\pi_Pd$ for $d\in L_{P,1}$. If $v_P(2na_2+a_4)=1$, the algorithm ends in step 4. We then have that $\delta_K(III,2,0;P)=\frac{Q_P-1}{Q_P^4}$.

Assume $2na_2+a_4$ is divisible by $\pi_P^2$. The total density for this case is $\frac{1}{Q_P^4}$. We have that $v_P(n^2 a_2+n a_4+a_6+n^3)=2$ with density $\frac{Q_P-1}{Q_P^5}$ from adding multiples of $\pi_P^2$ to $a_6$. If this is true, the algorithm ends in step 5. Afterwards, we have that $\delta_K(IV,1,0;P)=\delta_K(IV,3,0;P)=\frac{Q_P-1}{2Q_P^5}$.

Suppose $v_P(n^2 a_2+n a_4+a_6+n^3)\geq 3$. The total density for this case is $\frac{1}{Q_P^5}$. In step 6, there is no translation. Suppose $a_2$ is replaced by $a_2+d_1\pi_P$, $a_4$ is replaced with $a_4-2\alpha_1d_1\pi_P$, and $a_6$ is replaced with $a_6+\alpha_1^2d_1\pi_P$ for $d_1\in L_{P, 1}$. Note that the previous parts of the algorithm will not be changed. However, this changes the coefficient of $x^2$ from $a_2$ to $a_2+d_1\pi_P$, which changes the coefficient of $T^2$ of $P(T)$ in step 6. Next, replace $a_4$ with $a_4+d_2\pi_P^2$ and $a_6$ with $a_6-\alpha_1d_2\pi_P^2$ for $d_2\in \pi_P$. Similarly, this does not change the previous parts of the algorithm. However, $d_2\pi_P^2$  will be added to the coefficient of $x$, which adds $d_2$ to the coefficient of $T$ of $P(T)$. Afterwards, replace $a_6$ with $a_6+d_3\pi_P^3$ for $d_3\in L_{P, 1}$. This adds $d_3$ to the constant term $P(T)$. With this, the choices for $P(T)$ are the monic polynomials with degree $3$ in $(R_P/\pi_PR_P)[T]$; each choice for $P(T)$ corresponds to a density of $\frac{1}{Q_P^8}$. Moreover, the number of $P(T)$ with a double root and triple root are $Q_P(Q_P-1)$ and $Q_P$, respectively. 

Assume $P(T)$ has distinct roots. We have that the algorithm ends in step 6, with $\delta_K(I_0^*, 1, 0; P)=\frac{Q_P^2-1}{3Q_P^7}$, $\delta_K(I_0^*, 2, 0; P)=\frac{Q_P-1}{2Q_P^6}$, and $\delta_K(I_0^*,4,0;P)=\frac{Q_P^2-3Q_P+2}{6Q_P^7}$.

Assume $P(T)$ has a double root. For this case, Tate's algorithm ends in step 7 and the total density is $\frac{Q_P-1}{Q_P^7}$. In \Cref{subsec:subden2}, we compute that $\delta_K(I_N^*,2,0; P)=\delta_K(I_N^*,4,0; P)=\frac{(Q_P-1)^2}{2Q_P^{N+7}}$ for all positive integers $N$.

Next, assume $P(T)$ has a triple root. The density for this case is $\frac{1}{Q_P^7}$. Let $\alpha_2$ be the element of $L_{P,1}$ such that
\[
n^2a_2+na_4+a_6+n^3\equiv -\pi_P^3\alpha_2^3 \pmod{\pi_P^4}.
\]
Then, for the translation in step 8, we let $n=\alpha_1+\alpha_2\pi_P$. Suppose $v_P(n^2a_2+na_4+a_6+n^3)=4$. This occurs with density $\frac{Q_P-1}{Q_P^8}$ by adding multiples of $\pi_P^4$ to $a_6$. In this case, Tate's algorithm ends in step 8, and $\delta_K(IV^*,1,0;P)=\delta_K(IV^*,3,0;P)=\frac{Q_P-1}{2Q_P^8}$.

Assume $v_P(n^2a_2+na_4+a_6+n^3)\geq 5$. The total density for this case is $\frac{1}{Q_P^8}$. Consider replacing $a_4$ with $a_4+d\pi_P^3$ and $a_6$ with $a_6-(\alpha_1+\alpha_2\pi_P)d\pi_P^3$ for $d\in L_{P,1}$. This does not change previous parts of the algorithm but adds $d\pi_P^3$ to the coefficient of $x$. Therefore, $v_P(2na_2+a_4)=3$ with density $\frac{Q_P-1}{Q_P^9}$. For this, we have that Tate's algorithm ends in step 9 and $\delta_K(III^*,2,0;P)=\frac{Q_P-1}{Q_P^9}$.

Suppose $v_P(2na_2+a_4)\geq 4$. The total density of this case is $\frac{1}{Q_P^9}$. From adding multiples of $\pi_P^6$ to $a_6$, $v_P(n^3+a_2n^2+a_4n+a_6)=5$ with density $\frac{Q_P-1}{Q_P^{10}}$. Also, if $v_P(n^3+a_2n^2+a_4n+a_6)=5$, the algorithm ends in step 10. This gives that $\delta_K(II^*,1,0;P)=\frac{Q_P-1}{Q_P^{10}}$.

Similarly, the non-minimal curves have density $\frac{1}{Q_P^{10}}$.

\subsection{Subprocedure density calculations}
\label{subsec:subden2}

Let $X$ be the set of elliptic curves $E\in G_P^{(2)}$ such that $N_P(E)=0$ and Tate's algorithm enters the step 7 subprocedure when used on $E$. For $E\in X$, let $L(E)$ be the number of iterations of the step 7 subprocedure that are completed when Tate's algorithm is used on $E$. For a nonnegative integer $N$, let $X_N$ be the set of $E\in X$ such that $L(E)\geq N$. 

Suppose $N$ is an even nonnegative integer. Iteration $N$ of the step 7 subprocedure is completed if and only if $n\in R_P$ exists such that $v_P(a_2)=1$, $v_P(2n a_2+a_4)\geq\frac{N+6}{2}$, and $v_P(n^3+n^2a_2+na_4+a_6)\geq N+4$. Assume $n=n_1$ satisfies the condition. Suppose $n=n_2$ satisfies the condition also. Because $v_P(a_2)=1$, $v_P(n_1-n_2)\geq\frac{N+4}{2}$. Next, assume that $v_P(n_1-n_2)\geq\frac{N+4}{2}$. We show that $n=n_2$ also satisfies the condition. Clearly, $v_P(2n_2a_2+a_4)\geq\frac{N+6}{2}$. Moreover, we have that
\[
n_2^2a_2+n_2a_4 = n_1^2a_2+n_1a_4+\frac{1}{2}(n_2-n_1)((2n_1a_2+a_4)+(2n_2a_2+a_4)).
\]
Therefore, $v_P(n_2^3+n_2^2a_2+n_2a_4+a_6)\geq N+4$. We have that $n=n_2$ satisfies the condition if and only if $v_P(n_1-n_2)\geq \frac{N+4}{2}$.

Next, suppose $N$ is an odd positive integer. Iteration $N$ of the step 7 subprocedure is completed if and only if $n\in R_P$ exists such that $v_P(n^2 a_2 + na_4+a_6+n^3)\geq N+4$ and $v_P(2n a_2 + a_4)\geq\frac{N+5}{2}$. Assume $n=n_1$ satisfies the condition. Similarly to when $N$ is even, we have that $n=n_2$ also satisfies the condition if and only if $v_P(n_1-n_2)\geq \frac{N+3}{2}$.

Suppose $N$ is a nonnegative integer. Let $Y_N$ be the set of curves $y^2=x^3+a_2'x^2+a_4'x+a_6'$ with $v_P(a_2')=1$, $v_P(a_4')\geq\left\lfloor\frac{N+6}{2}\right\rfloor$, and $v_P(a_6')\geq\ N+4$. For $E\in X_N$, let $n_N(E)$ be the unique value of $n$ in $L_{P, \left\lfloor\frac{N+4}{2}\right\rfloor}$ from above. Suppose $\theta_N(E)$, with $\theta_N: X_N\rightarrow Y_N$, is the curve
\begin{align*}
\theta_N(E): y^2 & =(x+n_N(E))^3+a_2(x+n_N(E))^2+a_4(x+n_N(E))+a_6 \\
& = x^3 + a_2x^2 + (2n_N(E)a_2+a_4)x + n_N(E)^2 a_2 +n_N(E) a_4 +a_6.
\end{align*}

\begin{lemma}
\label{invdensitysub2}
If $U$ is an open subset of $Y_N$, $\mu_P(\theta_N^{-1}(U))=Q_P^{\left\lfloor\frac{N+4}{2}\right\rfloor}\mu_P(U)$.
\end{lemma}
\begin{proof}
Suppose $n\in L_{P, \left\lfloor\frac{N+4}{2}\right\rfloor}$. Let $X_{N,n}$ be the set of $E\in X_N$ with $n_N(E)=n$ and $\theta_{N,n}$ be $\theta_N$ restricted to $X_{N,n}$. Suppose $U$ is an open subset of $Y_N$. Using a method similar to the proof of \Cref{invdensitysub1}, we have that
\[
\mu_P(\theta_{N,n}^{-1}(U)) = \mu_P(U).
\]
Because there are $Q_P^{\left\lfloor\frac{N+4}{2}\right\rfloor}$ values of $n$, the result follows.
\end{proof}

Suppose $N$ is a positive integer. Using \Cref{invdensitysub2}, we can compute the density of the curves $E$ with $N_P(E)=0$ that have type $I_N^*$ and Tamagawa number $2$ or $4$. The Haar measure of the curves in $Y_{N-1}$ that end in iteration $N$ is $\frac{(Q_P-1)^2}{Q_P^{N+6+\left\lfloor\frac{N+5}{2}\right\rfloor}}$. With \Cref{invdensitysub2}, we have that $\delta_K(I_N^*, 2, 0; P)=\delta_K(I_N^*, 4, 0; P)=\frac{(Q_P-1)^2}{2Q_P^{N+7}}$; note that we can add multiples of $\pi_P^{N+3}$ to $a_6$ to deduce that $c=2$ and $c=4$ have the same density.

\section{Local Densities for \texorpdfstring{$p=2$}{}}
\label{sec:localc}

\subsection{Setup}

Assume that the characteristic of $K$ is $p=2$. Let $P$ be a place of $K$ and $G_P^{(3)}$ be the set of curves
\[
y^2+a_1xy+a_3y=x^3+a_4x+a_6
\]
over $K_P$ such that $a_1,a_3,a_4,a_6\in R_P$. Note that $G_P^{(3)}$ can be considered to be $R_P^4$. Define $\varphi: G_P\rightarrow G_P^{(3)}$ as the function such that if $E$ is the curve in $G_P$ with equation $E: y^2+a_1xy+a_3y=x^3+a_2x^2+a_4x+a_6$, $\varphi(E)$ is the curve in $G_P^{(3)}$ with equation
\[
\varphi(E):y^2+a_1xy+\left(a_3-\frac{a_1a_2}{3}\right)y=x^3+\left(a_4-\frac{a_2^2}{3}\right)x+\frac{2a_2^3}{27}-\frac{a_2a_4}{3}+a_6.
\]
Note that if $E$ is an elliptic curve, $E$ and $\varphi(E)$ are isomorphic.

\begin{lemma}
\label{invdensity3} 
If $U$ is an open subset of $G_P^{(3)}$, $\mu_P(\varphi^{-1}(U))=\mu_P(U)$.
\end{lemma}

\begin{proof}
This can be proved using a method similar to the proof of \Cref{invdensity1}.
\end{proof}

\subsection{Densities after multiple iterations of Tate's algorithm}

Let $k$ be a nonnegative integer. Suppose $S_k$ is the set of elliptic curves $E\in G_P^{(3)}$ such that $N_P(E)\geq k$.

For an elliptic curve $E\in G_P^{(3)}$ with equation $E:y^2+a_1xy+a_3y=x^3+a_4x+a_6$, let $A_k(E)$ be the set of $(l,m,n)\in R_P^3$ such that
\begin{align*}
& \left(y+\frac{l}{\pi_P^k}x+\frac{m}{\pi_P^{3k}}\right)^2+\frac{a_1}{\pi_P^k}\left(x+\frac{n}{\pi_P^{2k}}\right)\left(y+\frac{l}{\pi_P^k}x+\frac{m}{\pi_P^{3k}}\right) \\
& + \frac{a_3}{\pi_P^{3k}}\left(y+\frac{l}{\pi_P^k}x+\frac{m}{\pi_P^{3k}}\right) 
- \left(x+\frac{n}{\pi_P^{2k}}\right)^3-\frac{a_4}{\pi_P^{4k}}\left(x+\frac{n}{\pi_P^{2k}}\right)-\frac{a_6}{\pi_P^{6k}} \in R_P[x,y].
\end{align*}

\begin{theorem}
\label{prop:unique2}
Let $E$ be an elliptic curve in $G_P^{(3)}$. Then, $E\in S_k$ if and only if a unique pair $(l,m)\in L_{P,k}\times L_{P,3k}$ exists such that $(l,m,l^2+a_1l)\in A_k(E)$.
\end{theorem}

\begin{proof}
Suppose a unique pair $(l,m)$ satisfying the conditions exists. Because $A_k(E)$ is nonempty, $E\in S_k$ from \Cref{prop:multiterations}.

Assume $E\in S_k$. Then, using \Cref{prop:multiterations}, $A_k(E)$ is nonempty. Let the equation of $E$ be $E:y^2+a_1xy+a_3y=x^3+a_4x+a_6$ for $a_1,a_3,a_4,a_6\in R_P$.

From replacing $y$ with $y+l'x$ for $l'\in R_P$, if $(l,m,n)\in A_k(E)$, $(l+l'\pi_P^{k}, m,n)\in A_k(E)$. Therefore, there exist $l\in L_{P,k}$ and $m,n\in R_P$ such that $(l,m,n)\in A_k(E)$. Moreover, if $(l,m,n)\in A_k(E)$, $l^2+a_1l+n\equiv 0\pmod{\pi_P^{2k}}$. With this, from replacing $x$ with $x+\frac{l^2+a_1l+n}{\pi_P^{2k}}$, if $(l,m,n)\in A_k(E)$, $(l,m+l(l^2+a_1l+n),l^2+a_1l)\in A_k(E)$. Therefore, there exist $l\in L_{P,k}$ and $m\in R_P$ such that $(l,m,l^2+a_1l)\in A_k(E)$. Next, from replacing $y$ with $y+m'$ for $m'\in R_P$, there exists $l\in L_{P,k}$ and $m\in L_{P,3k}$ such that $(l,m,l^2+a_1l)\in A_k(E)$.

Next, we prove that $(l,m)$ is unique. Assume that $(l_1,m_1), (l_2,m_2)\in L_{P,k}\times L_{P,3k}$ and $(l_1,m_1,l_1^2+a_1l_1), (l_2,m_2,l_2^2+a_1l_2) \in A_k(E)$. We prove that $(l_1,m_1)=(l_2,m_2)$.

Let $F$ be the curve 
\[
F:y^2+\frac{a_1}{\pi_P^k}xy+\frac{a_3}{\pi_P^{3k}}y=x^3+\frac{a_4}{\pi_P^k}x+\frac{a_6}{\pi_P^{6k}}.
\]
For $1\leq i\leq 2$, let $F_i$ be $F$ with $x$ replaced by $x+\frac{l_i^2+a_1l_i}{\pi_P^{2k}}$ and $y$ replaced by $y+\frac{l_i}{\pi_P^{k}}x+\frac{m_i}{\pi_P^{3k}}$. Note that $F_i\in G_P^{(3)}$ because $(l_i, m_i, l_i^2+a_il_i)\in A_k(E)$ for $1\leq i\leq 2$. From this, $a_1\equiv 0\pmod{\pi_P^k}$. 

Suppose $a_1\not=0$. We have that $F_1$ and $F_2$ are isomorphic and $v_P(\Delta(F_1))=v_P(\Delta(F_2))$. Then, using \Cref{translation}, let $\tau$ be a translation from the equation of $F_1$ to the equation of $F_2$ that replaces $x$ with $u^2x+n'$ and $y$ with $u^3y+l'u^2x+m'$, where $u,l',m',n'\in R_P$ and $v_P(u)=0$. 

The coefficient of $xy$ after $\tau$ is applied to the equation of $F_1$ is $\frac{a_1}{u\pi_P^k}$. However, the coefficient of $xy$ in $F_2$ is $\frac{a_1}{\pi_P^k}$. Therefore, $u=1$ and $a_1\equiv 0\pmod{\pi_P^k}$. 

Next, the coefficient of $y$ after $\tau$ is applied to the equation of $F_1$ \[
\frac{a_1l_1^2+a_1^2l_1+a_3+\pi_P^{2k}a_1n'}{\pi_P^{3k}}.
\]
However, the coefficient of $y$ in $F_2$ is
\[
\frac{a_1l_2^2+a_1^2l_2+a_3}{\pi_P^{3k}}.
\]
Therefore, 
\[
l_1^2+a_1l_1+\pi_P^{2k}n'=l_2^2+a_1l_2.
\]
Because $a_1\equiv 0\pmod{\pi_P^k}$, we have that $l_1\equiv l_2\pmod{\pi_P^k}$. Therefore, $l_1=l_2$. From this, $n'=0$.

The coefficient of $x^2$ after $\tau$ is applied to the equation of $F_1$ is
\[
n'+(l')^2+\frac{a_1l'}{\pi_P^k}.
\]
This equals the coefficient of $x^2$ in $F_2$, which is $0$. Because $n'=0$, we have that $l'=0$ or $l'=\frac{a_1}{\pi_P^k}$.

From setting the coefficient of $x$ after $\tau$ is applied to the equation of $F_1$ equal to the coefficient of $x$ in $F_2$,
\[
\frac{a_1}{\pi_P^k}\cdot\left(\frac{m_1}{\pi_P^{3k}}+m'\right)+\frac{a_1(l_1^2+a_1l_1)+a_3}{\pi_P^{3k}}\cdot l'=\frac{a_1}{\pi_P^k}\cdot\frac{m_2}{\pi_P^{3k}}.
\]
Suppose $l'=0$. Then $\frac{m_1}{\pi_P^{3k}}+m'=\frac{m_2}{\pi_P^{3k}}$. It follows that $m_1\equiv m_2\pmod{\pi_P^{3k}}$ and $m_1=m_2$. Suppose $l'=\frac{a_1}{\pi_P^k}$. We have that
\[
\frac{m_1}{\pi_P^{3k}}+m'+\frac{a_1(l_1^2+a_1l_1)+a_3}{\pi_P^{3k}}=\frac{m_2}{\pi_P^{3k}}.
\]
However, using that the coefficient of $y$ in $F_2$ is an element of $R_P$,
\[
a_1(l_1^2+a_1l_1)+a_3\equiv a_1(l_2^2+a_1l_2)+a_3\equiv 0\pmod{\pi_P^{3k}}.
\]
Therefore, $m_1\equiv m_2\pmod{\pi_P^{3k}}$ and $m_1=m_2$.

Assume $a_1=0$. From the coefficient of $y$ in $F_2$, we have that $a_3\equiv 0\pmod{\pi_P^{3k}}$. Also, from the coefficients of $x$ in $F_1$ and $F_2$,
$l_1^4+a_3l_1\equiv l_2^4+a_3l_2\pmod{\pi_P^{4k}}$.
This gives that $l_1=l_2$. Afterwards, from the constant terms of $F_1$ and $F_2$, $m_1^2+a_3m_1\equiv m_2^2+a_3m_2\pmod{\pi_P^{6k}}$. From this, we obtain that $m_1=m_2$.
\end{proof}

Using \Cref{prop:unique2}, for $E\in S_k$, let the unique pair $(l,m)\in L_{P,k}\times L_{P,3k}$ such that $(l,m,l^2+a_1l)\in A_k(E)$ be $(l(E), m(E))$. Define $\phi_k: S_k\rightarrow S_0$ to be the function such that if $E\in S_k$ has equation $E:y^2+a_1xy+a_3y=x^3+a_4x+a_6$, $\phi_k(E)$ has equation
\begin{align*}
& \phi_k(E): y^2+\frac{a_1}{\pi_P^k}xy+\frac{a_1(l(E)^2+a_1l(E))+a_3}{\pi_P^{3k}} = x^3+\\
& \frac{l(E)^2(l(E)^2+a_1l(E))+a_1m(E)+a_3l(E)+a_4}{\pi_P^{4k}}x + \\
& \frac{(a_1m(E)+a_4+a_1^2l(E)^2+l(E)^4)(l(E)^2+a_1l(E))+a_3m(E)+a_6+m(E)^2}{\pi_P^{6k}}.
\end{align*}
The equation for $\phi_k(E)$ is equivalent to
\begin{align*}
& \left(y+\frac{l(E)}{\pi_P^k}x+\frac{m(E)}{\pi_P^{3k}}\right)^2+\frac{a_1}{\pi_P^k}\left(x+\frac{l(E)^2+a_1l(E)}{\pi_P^{2k}}\right)\left(y+\frac{l(E)}{\pi_P^k}x+\frac{m(E)}{\pi_P^{3k}}\right) \\ & +\frac{a_3}{\pi_P^{3k}}\left(y+\frac{l(E)}{\pi_P^k}x+\frac{m(E)}{\pi_P^{3k}}\right) = \\
& \left(x+\frac{l(E)^2+a_1l(E)}{\pi_P^{2k}}\right)^3+\frac{a_4}{\pi_P^{4k}}\left(x+\frac{l(E)^2+a_1l(E)}{\pi_P^{2k}}\right)+\frac{a_6}{\pi_P^{6k}}.
\end{align*}
Note that $S_0\subset G_P^{(3)}$, and from \Cref{densityec} and \Cref{invdensity3}, $\mu_P(S_0)=1$.  For $l\in L_{P, k}$ and $m\in L_{P, 3k}$, let $S_{k,l,m}$ be the set of $E\in S_k$ such that $l(E)=l$ and $m(E)=m$. Assume that $\phi_{k,l,m}$ is $\phi_k$ restricted to $S_{k,l,m}$.

\begin{lemma}
\label{invdensitymult3a}
Suppose $l\in L_{P,k}$ and $m\in L_{P, 3k}$. If $U$ is an open subset of $G_P^{(3)}$, $\mu_P(\phi_{k,l,m}^{-1}(U))\\=\frac{1}{Q_P^{14k}}\mu_P(U)$.
\end{lemma}
\begin{proof}
This can be proved with a method that is similar to the proof of \Cref{invdensitymult2a}.
\end{proof}

\begin{lemma} 
\label{invdensitymult3b}
If $U$ is an open subset of $G_P^{(3)}$, $\mu_P(\phi_k^{-1}(U))=\frac{1}{Q_P^{10k}}\mu_P(U)$.
\end{lemma}

\begin{proof}
Let $U$ be an open subset of $G_P^{(3)}$. We have that $\phi_k^{-1}(U)=\bigsqcup_{l\in L_{P,k}, m\in L_{P,3k}}\phi_{k,l,m}^{-1}(U)$. Using \Cref{invdensitymult3a},
\begin{align*}
\mu_P(\phi_k^{-1}(U))& =\sum_{l\in L_{P,k}}\sum_{m\in L_{P,3k}}\mu_P(\phi_{k,l,m}^{-1}(U))= \sum_{l\in L_{P,k}}\sum_{m\in L_{P,3k}} \frac{1}{Q_P^{14k}}\mu_P(U)= \frac{1}{Q_P^{10k}}\mu_P(U),
\end{align*}
completing the proof.
\end{proof}

\subsection{Density calculations for \texorpdfstring{$v_P(a_1)=0$}{}}

Suppose that $v_P(a_1)=0$. This case has density $\frac{Q_P-1}{Q_P}$. The discriminant is
\[
a_1^4(a_1^2a_6+a_1a_3a_4+a_4^2)+a_3^4+a_1^3a_3^3.
\]

Note that by considering $a_6$ modulo $\pi_P$, the discriminant is not divisible by $\pi_P$ with density $\frac{(Q_P-1)^2}{Q_P^2}$. For this case, the algorithm ends in step 1. Then, we add $\frac{(Q_P-1)^2}{Q_P^2}$ to $\delta_K(I_0,1,0;P)$.

Assume the discriminant is divisible by $\pi_P$. Let $(\alpha_1, \alpha_2)$ be the singular point modulo $\pi_P$; it can be proven that $\alpha_1, \alpha_2 \in R_P$. Also, $\alpha_1\equiv-\frac{a_3}{a_1}\pmod{\pi_P}$. In step 2, replace $x$ by $x+n$ and $y$ by $y+m$ with $n=\alpha_1$ and $m=\alpha_2$. Afterwards, the coefficient of $xy$ is $a_1$, which is not divisible by $\pi_P$. The algorithm then ends in step 2. 

We see that the discriminant is linear in $a_6$. Therefore, we have that $v_P(a_1)=0$ and $v_P(\Delta(E))=N$ with density $\frac{(Q_P-1)^2}{Q_P^{N+2}}$ for $N\geq 0$. Note that the polynomial considered in step 2 is $T^2+a_1T+\alpha_1$. Suppose $a_1\equiv r_1\pmod{\pi_P}$ and $a_3\equiv r_3\pmod{\pi_P}$ for $r_1, r_3\in L_{P,1}$ such that $r_1\not=0$. Given $r_1$, $T^2+a_1T+\alpha_1$ is irreducible over $R_P/\pi_PR_P$ for $\frac{Q_P}{2}$ values of $r_3$. Afterwards, using step 2 of Tate's algorithm, we get that in this case, $\delta_K(I_1,1,0;P)=\frac{(Q_P-1)^2}{Q_P^3}$, $\delta_K(I_2,2,0;P)=\frac{(Q_P-1)^2}{Q_P^4}$, and 
\[
\delta_K(I_N,N,0;P)=\delta_K\left(I_N,2\left\lfloor\frac{N}{2}\right\rfloor-N+2,0;P\right)=\frac{(Q_P-1)^2}{2Q_P^{N+2}}
\]
for $N\geq 3$. 

\subsection{Density calculations for \texorpdfstring{$v_P(a_1)\geq 1$}{}}

In this subsection, we assume that $v_P(a_1)\geq 1$. The density for this is $\frac{1}{Q_P}$ and the discriminant modulo $\pi_P$ is $a_3^4$. 

Suppose $v_P(a_3)=0$. The density for this case is $\frac{Q_P-1}{Q_P^2}$ and the discriminant is not divisible by $\pi_P$. Tate's algorithm then ends in step 1 and we add $\frac{Q_P-1}{Q_P^2}$ to $\delta_K(I_0,1,0;P)$. Following this, we obtain that $\delta_K(I_0,1,0;P)=\frac{Q_P-1}{Q_P}$.

Next, assume $v_P(a_3)\geq 1$. The total density for this case is $\frac{1}{Q_P^2}$. The singular point modulo $\pi_P$ is $(x,y)=(\alpha_1, \alpha_2)$ for $\alpha_1, \alpha_2\in L_{P, 1}$ such that $a_4\equiv\alpha_1^2\pmod{\pi_P}$ and $a_6\equiv\alpha_2^2\pmod{\pi_P}$. We replace $x$ with $x+n$ and $y$ with $y+m$, where $n=\alpha_1$ and $m=\alpha_2$. The curve is
\[
(y+m)^2+a_1(x+n)(y+m)+a_3(y+m)=(x+n)^3+a_4(x+n)+a_6.
\]

If $\pi_P^2$ does not divide $mna_1+ma_3+na_4+a_6+m^2+n^3$, the algorithm ends in step 3. By adding multiples of $\pi_P$ to $a_6$, this occurs with density $\frac{Q_P-1}{Q_P^3}$. We have that $\delta_K(II,1,0;P)=\frac{Q_P-1}{Q_P^3}$.

Assume $\pi_P^2$ divides $mna_1+ma_3+na_4+a_6+m^2+n^3$. The total density for this case is $\frac{1}{Q_P^3}$. We have that
\[
b_8=n(n a_1+a_3)^2+(m a_1+a_4+n^2)^2.
\]
If $b_8$ is not divisible by $\pi_P^3$, the algorithm ends in step 4. By adding multiples of $\pi_P$ to $a_4$, we have that $\delta_K(III,2,0;P)=\frac{Q_P-1}{Q_P^4}$.

Assume $b_8$ is divisible by $\pi_P^3$. The total density for this case is $\frac{1}{Q_P^4}$. If $v_P(n a_1+a_3)=1$, the algorithm ends in step 5. Assume $a_4\equiv 0\pmod{\pi_P}$. Then, replace $a_3$ with $a_3+d\pi_P$ and $a_4$ with $a_4+\beta d\pi_P$ for $\beta, d\in L_{P,1}$ such that $\beta^2\equiv\alpha_1\pmod{\pi_P}$. This will not affect the previous steps of the algorithm; particularly, this will not change $b_8$ modulo $\pi_P^3$. However, $na_1+a_3$ will be increased by $d\pi_P$. Therefore, we have that $v_P(n a_1+a_3)=1$ with density $\frac{Q_P-1}{Q_P^5}$. From this, $\delta_K(IV,1,0;P)=\delta_K(IV,3,0;P)=\frac{Q_P-1}{2Q_P^5}$.

Assume $v_P(n a_1+a_3)\geq 2$. The total density for this case is $\frac{1}{Q_P^5}$. Let $\alpha_3$ be the element of $L_{P, 1}$ such that $n\equiv\alpha_3^2\pmod{\pi_P}$. Also, let $\alpha_4$ be the element of $L_{P, 1}$ such that
$mna_1+ma_3+na_4+a_6+m^2+n^3\equiv \alpha_4^2\pi_P^2\pmod{\pi_P^3}$. After the transformation in step 6, the equation of the curve is
\begin{align*}
& (y+l x+m)^2+a_1(x+n)(y+l x+m)+a_3(y+l x+m)  =(x+n)^3+a_4(x+n)+a_6,
\end{align*}
where $n=\alpha_1$, $l=\alpha_3$, and $m=\alpha_2+\alpha_4\pi_P$. Suppose that in step 6, the polynomial $P(T)\in (R_P/\pi_PR_P)[T]$ is $P(T)=T^3+w_2T^2+w_1T+w_0$.

Suppose $a_4\equiv 0\pmod{\pi_P}$. Because $0\in L_{P, 1}$, we have that $n=l=0$. This means that $w_2=0$. Then, we can replace $a_4$ with $a_4+d_1\pi_P^2$ for $d_1\in L_{P,1}$ and the previous steps of the algorithm will not be changed. With this, the choices for $w_1$ modulo $\pi_P$ are the elements of $L_{P,1}$. Following this, from replacing $a_6$ with $a_6+d_2\pi_P^3$ for $d_2\in L_{P,1}$, the choices for $w_0$ modulo $\pi_P$ are the elements of $L_{P,1}$. We have that the number of $P(T)$ with a double root and no roots in $\overline{R_P/\pi_PR_P}$ are $Q_P-1$ and $1$, respectively. Moreover, we have that the number of $P(T)$ with $3$ distinct roots in $\overline{R_P/\pi_PR_P}$ and $0$ roots, $1$ root, and $3$ roots in $R_P/\pi_PR_P$ are $\frac{Q_P^2-1}{3}$, $\frac{Q_P^2-Q_P}{2}$, and $\frac{Q_P^2-3Q_P+2}{6}$, respectively.

Suppose $a_4\not\equiv 0\pmod{\pi_P}$. Consider the translation that replaces $a_1$ with $a_1+d_1\pi_P$, $a_3$ with $a_3+\alpha_1d_1\pi_P$, $a_4$ with $a_4+(\alpha_2+\alpha_4\pi_P)d_1\pi_P$, and $a_6$ with $a_6+\alpha_1(\alpha_2+\alpha_4\pi_P) d_1\pi_P$ for $d_1\in L_{P,1}$. After this, the steps of the algorithm before step 6 do not change. In step 6, $w_0$ and $w_1$ do not change. However, $w_2$ increases by $\alpha_3d_1$. Because $\alpha_3\not=0$, the choices for $w_2$ are the elements of $L_{P,1}$. Next, replace $a_6$ with $a_6+d_2\pi_P^3$ for $d_2\in L_{P,1}$. With this, the choices for $w_0$ are also the elements of $L_{P,1}$. The number of $P(T)$ with a double root and no roots in $\overline{R_P/\pi_PR_P}$ are the same as above. Furthermore, the number of $P(T)$ with $3$ distinct roots in $\overline{R_P/\pi_PR_P}$ and $0$ roots, $1$ root, and $3$ roots in $R_P/\pi_PR_P$ are the same as above. 

Suppose $P(T)$ has distinct roots. For this case, the total density is $\frac{Q_P-1}{Q_P^6}$ and Tate's algorithm ends in step 6. We see that $\delta_K(I_0^*,1,0;P)=\frac{Q_P^2-1}{3Q_P^7}$, $\delta_K(I_0^*,2,0;P)=\frac{Q_P-1}{2Q_P^6}$, and $\delta_K(I_0^*,4,0;P)=\frac{Q_P^2-3Q_P+2}{6Q_P^7}$.

Assume that $P(T)$ has a double root and a simple root. For this case, the total density is $\frac{Q_P-1}{Q_P^7}$ and Tate's algorithm ends in step 7. In \Cref{subsec:subden3}, we compute that $\delta_K(I_N^*, 2,0 ;P)=\delta_K(I_N^*,4,0;P)=\frac{(Q_P-1)^2}{2Q_P^{N+7}}$ for all positive integers $N$.

Next, suppose $P(T)$ has a triple root. For this case, the density is $\frac{1}{Q_P^7}$ and the root of $P(T)$ is $\sqrt{w_1}$ modulo $\pi_P$. If $a_4\equiv 0\pmod{\pi_P}$, the triple root is $0$ modulo $\pi_P$. Let $\alpha_5$ be an element of $L_{P,1}$ such that
\[(m+ln)a_1 + la_3+a_4+n^2\equiv\alpha_5^2\pi_P^2\pmod{\pi_P^3}.
\]
Then, the translation in step 8 sets $n$ to be $n=\alpha_1+\alpha_5\pi_P$.

Suppose $a_4\equiv 0\pmod{\pi_P}$. Replace $a_3$ with $a_3+d\pi_P^2$ and $a_6$ with $a_6+(\alpha_2+\alpha_4\pi_P) d\pi_P^2$ for some $d\in L_{P,1}$. Then, note that the previous parts of the algorithm, including $P(T)$, are unchanged. However, the coefficient of $y$ increases by $d\pi_P^2$. We have that for one value of $d$, the coefficient of $y$ is divisible by $\pi_P^3$. Next, suppose $a_4\not\equiv 0\pmod{\pi_P}$. Replace $a_1$ with $a_1+d\pi_P^2$ and $a_4$ with $a_4+(\alpha_2+\alpha_4\pi_P)d\pi_P^2$ for some $d\in L_{P,1}$. The previous parts of the algorithm, including $P(T)$, are unchanged. However, the coefficient of $y$ increases by $(\alpha_1+\alpha_5\pi_P)d\pi_P^2$. Similarly, we have that for one value of $d$, the coefficient of $y$ is divisible by $\pi_P^3$. From this, we get that the coefficient of $y$ is not divisible by $\pi_P^3$ and the algorithm ends in step 8 with density $\frac{Q_P-1}{Q_P^8}$. Thus, $\delta_K(IV^*,1,0;P)=\delta_K(IV^*,3,0;P)=\frac{Q_P-1}{2Q_P^8}$.

Assume the coefficient of $y$ is divisible by $\pi_P^3$. The total density of this case is $\frac{1}{Q_P^8}$. Let $\alpha_6$ be the element of $L_{P,1}$ such that
\[
mna_1+ma_3+na_4+a_6+m^2+n^3\equiv\alpha_6^2\pi_P^4\pmod{\pi_P^5}.
\]
Then, $m$ is set to $m=\alpha_2+\alpha_4\pi_P+\alpha_6\pi_P^2$ in step 9. If $\pi_P^4$ does not divide the $x$ coefficient of this curve, the algorithm ends in step 9. Consider the translation of replacing $a_4$ with $a_4+d\pi_P^3$ and $a_6$ with $a_6+(\alpha_1+\alpha_5\pi_P)d\pi_P^3$ for $d\in L_{P,1}$. The previous steps of the algorithm do not change but the coefficient of $x$ is increased by $d\pi_P^3$. Therefore, $\pi_P^4$ does not divide the $x$ coefficient with density $\frac{Q_P-1}{Q_P^9}$. We have that $\delta_K(III^*,2,0;P)=\frac{Q_P-1}{Q_P^9}$

Assume $\pi_P^4$ divides the coefficient of $x$ of the curve. The total density for this case is $\frac{1}{Q_P^9}$. If $\pi_P^6$ does not divide 
$mna_1+ma_3+na_4+a_6+m^2+n^3$, Tate's algorithm ends in step 10. This occurs with density $\frac{Q_P-1}{Q_P^{10}}$ from adding multiples of $\pi_P^6$ to $a_6$. We then have that $\delta_K(II^*,1,0;P)=\frac{Q_P-1}{Q_P^{10}}$.

Similarly, the non-minimal curves have density $\frac{1}{Q_P^{10}}$.

\subsection{Subprocedure density calculations}

\label{subsec:subden3}

We calculate the densities of the Kodaira types $\mathfrak{r}=I_N^*$ for $N\geq 1$ and the Tamagawa numbers $c=2,4$. Note that previously, the curve was reduced by removing $a_2$ with a translation on $x$ to obtain $G_P^{(3)}$. However, here the density is calculated in $G_P$ without the reduction. That is, the density is calculated for curves in long Weierstrass form.

Let $X$ be the set of elliptic curves $E\in G_P$ such that $N_P(E)=0$ and Tate's algorithm enters the step 7 subprocedure when it is applied to $E$. For $E\in X$, let $L(E)$ be the number of iterations of the step 7 subprocedure that are completed when Tate's algorithm is applied to $E$. For a nonnegative integer $N$, let $X_N$ be the set of $E\in X$ such that $L(E)\geq N$.

Suppose $N$ is an even nonnegative integer. Assume that $N=0$. In iteration $N=0$, there is a translation. Note that the double root of $P(T)$ is the square root of $w_1$. Because of this, in step 7, we add $\gamma_0\pi_P$ to $n$ and $l\gamma_0\pi_P$ to $m$ for some $\gamma_0\in L_{P,1}$ such that
\[
(m+l n) a_1 + l a_3+a_4+n^2\equiv \gamma_0^2\pi_P^2\pmod{\pi_P^3}
\]
Next, assume that $N\geq 2$. Suppose iteration $N$ of the step 7 subprocedure is reached and the quadratic has a double root. Then,
\[
v_P((m+ln)a_1+l a_3+a_4+n^2)\geq\frac{N+6}{2}.
\]
Also, we add $\gamma_N\pi_P^{\frac{N+2}{2}}$ to $n$ and $l\gamma_N\pi_P^{\frac{N+2}{2}}$  to $m$ for some $\gamma_N\in L_{P,1}$ such that
\[
mna_1+ma_3+na_4+a_6+m^2+n^3\equiv (la_1+a_2+n+l^2)\gamma_N^2\pi_P^{N+2}\pmod{\pi_P^{N+4}}.
\]
Note that $v_P(la_1+a_2+n+l^2)=1$. 

Suppose $N$ is an odd nonnegative integer. Suppose iteration $N$ of the step 7 subprocedure is reached and the quadratic has a double root. Then, $v_P(na_1+a_3)\geq\frac{N+5}{2}$. Also, $\gamma_N\pi_P^{\frac{N+3}{2}}$ is added to $m$ for some $\gamma_N\in L_{P,1}$ such that
\[
mna_1+ma_3+na_4+a_6+m^2+n^3\equiv \gamma_N^2\pi_P^{N+3}\pmod{\pi_P^{N+4}}
\]

Let $N$ be a nonnegative integer. Let $Y_N$ be the set of curves $y^2+a_1'xy+a_3'y=x^3+a_2'x^2+a_4'x+a_6'$ with $v_P(a_1')\geq 1$, $v_P(a_2')=1$, $v_P(a_3')\geq\lfloor\frac{N+5}{2}\rfloor, v_P(a_4')\geq \lfloor\frac{N+6}{2}\rfloor$, and $v_P(a_6')\geq N+4$.

Suppose $E\in X_N$ and that the translations of Tate's algorithm when it is used on $E$ are $\alpha_1$, $\alpha_2$, $\alpha_3$, $\alpha_4$, $\gamma_0$, $\gamma_1$, $\ldots$, $\gamma_N$. Let $T_N(E)=(\alpha_1,\alpha_2,\alpha_3,\alpha_4,\gamma_0,\gamma_1,\ldots,\gamma_N)$. Note that because the characteristic of $K$ is $p=2$, $T_N(E)$ is well defined. Also, let $\theta_N(E): X_N\rightarrow Y_N$ be $E$ with $x$ replaced by $x+n$ and $y$ replaced by $y+l x+m$, where 
\[
    n = \alpha_1+\sum_{i=0}^{\left\lfloor\frac{N}{2}\right\rfloor}\gamma_{2i}\pi_P^{i+1}, l = \alpha_3, m = \alpha_2+\alpha_4\pi_P+\alpha_3\sum_{i=0}^{\left\lfloor\frac{N}{2}\right\rfloor}\gamma_{2i}\pi_P^{i+1}+\sum_{i=0}^{\left\lfloor\frac{N-1}{2}\right\rfloor}\gamma_{2i+1}\pi_P^{i+2}.
\]

\begin{lemma}
\label{invdensitysub3}
If $U$ is an open subset of $Y_N$, $\mu_P(\theta_N^{-1}(U))=Q_P^{N+5}\mu_P(U)$.
\end{lemma}

\begin{proof}
Let $a=(\alpha_1,\alpha_2,\alpha_3,\alpha_4, \gamma_0, \gamma_1,\ldots, \gamma_N)_{0\leq i\leq N}$ be an element of $L_{P,1}^{N+5}$. Suppose that $X_{N,a}$ is the set of $E\in X_N$ such that $T_N(E)=a$. Suppose that $\theta_{N,a}$ is $\theta_N$ restricted to $X_{N,a}$. Let $U$ be an open subset of $Y_N$. Using a method similar to the proof of \Cref{invdensitysub1}, we have that
\[
\mu_P(\theta_{N,a}^{-1}(U)) = \mu_P(U).
\]
Because there are $Q_P^{N+5}$ choices of $a$, the result follows.
\end{proof}

Suppose $N$ is a positive integer. With \Cref{invdensitysub3}, we can compute the density for curves that enter step 7 in the first iteration and have type $I_N^*$. We have that $\mu_P(Y_{N-1})=\frac{Q_P-1}{Q_P^{2N+10}}$, and the Haar measure in $G_P^{(3)}$ of curves that have type $I_N^*$ is then $\frac{(Q_P-1)^2}{Q_P^{N+7}}$. Particularly, $\delta_K(I_N^*, 2, 0; P)=\delta_K(I_N^*, 4, 0; P)=\frac{(Q_P-1)^2}{2Q_P^{N+7}}$; note that we can add multiples of $\pi_P^{N+3}$ to $a_6$ to deduce that $c=2$ and $c=4$ have the same density.

\section{Local and Global Density Results}
\label{sec:results}

In Sections \ref{sec:locala}, \ref{sec:localb}, and \ref{sec:localc}, we compute the local densities of Koidara types and Tamagawa numbers for $p\geq 5$, $p=3$, and $p=2$, respectively. The methods we use involved first removing some terms from the equations of elliptic curves with translations, and then using translations to compute the local densities. 

Next, we discuss some results about local and global densities, including a proof of \Cref{thm:multiterdensity}. Particularly, we compute the density of completing at most $k\geq 0$ iterations of Tate's algorithm. 

\subsection{Proof of \texorpdfstring{\Cref{thm:multiterdensity}}{}}
\label{subsec:multiterproof}

In the following proof, the functions $\varphi$ and $\phi_k$ are defined in Sections \ref{sec:locala}, \ref{sec:localb}, and \ref{sec:localc} for each choice of the characteristic of $K$. Let $U$ and $V$ be the sets of elliptic curves $E\in G_P$ with Kodaira type $\mathfrak{r}$ and Tamagawa number $n$ such that $N_P(E)=0$ and $N_P(E)=k$, respectively. Note that $\varphi(U)$ and $\varphi(V)$ are the sets of curves $E\in S_0$ with Kodaira type $\mathfrak{r}$ and Tamagawa number $n$ such that $N_P(E)=0$ and $N_P(E)=k$, respectively.

Suppose $E\in G_P$ and $\varphi(E)\in \varphi(U)$. Then, $E$ has Kodaira type $\mathfrak{r}$, Tamagawa number $n$, and $N_P(E)=0$. This means that $E\in U$. From this, $\varphi^{-1}(\varphi(U))\subset U$. Moreover, $U\subset \varphi^{-1}(\varphi(U))$. It follows that $\varphi^{-1}(\varphi(U))=U$. Similarly, $\varphi^{-1}(\varphi(V))=V$.

We have that $U$ and $V$ are open sets. Moreover, $\varphi(U)$ and $\varphi(V)$ are open sets. With this, we have that $\mu_P(U)=\mu_P(\varphi(U))$ and $\mu_P(V)=\mu_P(\varphi(V))$ for all characteristics $p$ from \Cref{invdensity1,invdensity2,invdensity3}. Note that the image of $\varphi$ is $G_P^{(1)}$, $G_P^{(2)}$, or $G_P^{(3)}$ depending on the characteristic of $K$ and the densities of $\varphi(U)$ and $\varphi(V)$ are computed with respect to these sets. Therefore, it suffices to prove that
\[
\mu_P(\varphi(V))=\frac{1}{Q_P^{10k}}\mu_P(\varphi(U)).
\]

Suppose $E\in\varphi(V)$. We have that $\phi_k(E)$ has Kodaira type $\mathfrak{r}$, Tamagawa number $n$, and $N_P(\phi_k(E))=0$. Therefore, $\phi_k(E)\subset\varphi(U)$. It follows that $\varphi(V)\subset \phi_k^{-1}(\varphi(U))$. Next, suppose $E\in S_k$ and $\phi_k(E)\in \varphi(U)$. Then, the Koidara type of $E$ is $\mathfrak{r}$ and the Tamagawa number of $E$ is $n$. Moreover, because $N_P(\phi_k(E))=0$, $N_P(E)=k$. It follows that $E\in \varphi(V)$. Therefore, $\phi_k^{-1}(\varphi(U))\subset \varphi(V)$. From this, $\phi_k^{-1}(\varphi(U))=\varphi(V)$. The result then follows from \Cref{invdensitymult1,invdensitymult2b,invdensitymult3b}.

\subsection{Densities after multiple iterations of Tate's algorithm}

Let $k$ be a nonnegative integer. For $P\in M_K$, let $U_P^k$ denote the set of elliptic curves $E$ in $G_P$ such that $N_P(E)\geq k+1$. The following proposition is important for the proof of \Cref{multiterglobal}.

\begin{proposition}
\label{multiterlocal}
For $P\in M_K$, $\mu_P(U_P^k)=\frac{1}{Q_P^{10(k+1)}}$.
\end{proposition}

\begin{proof}
Suppose $P\in M_k$. From \Cref{invdensitymult1,invdensitymult2b,invdensitymult3b} with $k+1$ as $k$ and $G_P$ as $U$, we have that
\[
\mu_P(U_P^k)=\frac{1}{Q_P^{10(k+1)}}\cdot\mu_P(G_P)=\frac{1}{Q_P^{10(k+1)}}.
\]
This finishes the proof.
\end{proof}

\begin{theorem}
\label{multiterglobal}
Let $S$ be a finite nonempty subset of $M_K$. Suppose $U$ is the set of elliptic curves in $W_S$ such that $N_P(E)\leq k$ for all $P\in S^C$. Then,
\[
d_S(U)=\frac{1}{\zeta_K(10(k+1))}\prod_{P\in S} \left(\frac{Q_P^{10(k+1)}}{Q_P^{10(k+1)}-1}\right).
\]
\end{theorem}

\begin{proof}
For a positive integer $M$, let $V_M$ denote the set of elliptic curves $E\in W_S$ such that there exists $P\in S^C$ with degree at least $M$ such that $E\in U_P^k$. From \Cref{multitercond}, we have that $\lim_{M\rightarrow\infty}\overline{d}_S(V_M)=0$. Therefore, we can use \Cref{globaldensity1} with $U_P$ set as $U_P^k$ for $P\in S^C$ and $T=\{\}$. The result follows from \Cref{multiterlocal}.
\end{proof}

\begin{exmp}
We give an example of \Cref{multiterglobal}. Let $K=\mathbb{F}_q(t)$. Suppose $P_\infty$ is the infinite place of $\mathbb{F}_q(t)$ and let $S=\{P_\infty\}$. Let $k$ be a nonnegative integer and $U$ be the set of elliptic curves in $W_S$ such that $N_P(E)\leq k$ for all $P\in S^C$. From \cite{NTFunc}*{Theorem 5.9}, because the genus of $K$ is $0$, we have that $\zeta_K(10(k+1))=\frac{q^{20k+19}}{(q^{10k+9}-1)(q^{10k+10}-1)}$. Since $P_\infty$ has degree $1$, \Cref{multiterglobal} implies that $d_S(U)=1-\frac{1}{q^{10k+9}}$.
\end{exmp}

\subsection{The densities of the global Tamagawa numbers}
\label{subsec:tamagawanum}

In this subsection, $E\in W_S$ has \textit{global Tamagawa number} equal to $n$ if the product of $c_P(E)$ for $P\in S^C$ equals $n$. We are particularly interested in the case $n=1$, see \Cref{thm:existence}.

\begin{lemma}
\label{lemma:discsq}
Suppose $P\in M_K$ and $E\in G_P$. If $c_P(E)>1$, then $v_P(\Delta(E))\geq 2$.
\end{lemma}

\begin{proof}
This is straightforward to check using the computations from Sections \ref{sec:locala}, \ref{sec:localb}, and \ref{sec:localc}.
\end{proof}

\begin{proof}[Proof of \Cref{thm:densitytamn}]
Suppose we are given values of $c_P$ for $P\in S^C$ such that $\prod_{P\in S^C} c_P=n$. For a positive integer $M$, let $V_M$ be the set of elliptic curves $E\in W_S$ such that for some $P\in S^C$ with degree at least $M$, $c_P(E)\not=c_P$. Since only finitely many of the $c_P$ are greater than $1$ and $\Delta(E)$ is a square-free polynomial which must be divisible by $\pi_P^2$ if $c_P(E)>1$ by \Cref{lemma:discsq}, applying \Cref{lemma:sqfree} implies that $\lim_{M\rightarrow\infty} \overline{d}_S(V_M)=0$. Then, applying \Cref{globaldensity1} with $T$ as the empty set gives that the density of $E$ such that $c_P(E)=c_P$ for all $P\in S^C$ is $\prod_{P\in S^C} d_P(c_P)$. Summing over the choices for $c_P$ for $P\in S^C$ gives the result.
\end{proof}

\begin{lemma}
\label{lemma:globaldensitytam1} Over $W_S$, the density of minimal curves with global Tamagawa number $1$ is at least $\zeta_K(2)^{-1}$.
\end{lemma}

\begin{proof}
Suppose $P\in S^C$. Using \Cref{thm:formula} gives that the local density of the minimal curves with Tamagawa number $1$ at $P$ is at least
\[
\delta_K(I_0, 1, 0; P) + \delta_K(I_1, 1, 0; P) + \delta_K(II, 1, 0; P) = 1-\frac{1}{Q_P^2}.
\]
Hence, using \Cref{thm:densitytamn} implies that the density of the global Tamagawa number $1$ is at least $\prod_{P\in S^C} \left(1-\frac{1}{Q_P^2}\right) \geq \zeta_K(2)^{-1}$.
\end{proof}

The following result is an analogue of \cite{ProdNF}*{Theorem 1.5}.

\begin{theorem}
\label{thm:existence}
Suppose $\delta<1$. There exists a global function field $K$ such that the density of the global Tamagawa number $1$ over $W_S$ is at least $\delta$ for all finite nonempty $S\subset M_K$.
\end{theorem}

\begin{proof}
See \Cref{lemma:globaldensitytam1} and \Cref{thm:closeto1}.
\end{proof}

\section{Constructions of global function fields}
\label{sec:construction}

First, we discuss two formulations of the zeta function of the global function field $K$ from \cite{NTFunc}*{Chapter 5}, see \pref{eq:zetadefsum} and \pref{eq:zetadefprod}. For a nonnegative integer $n$, let $b_n(K)$ denote the number of divisors $D$ of $K$ such that $D\geq 0$ and the degree of $D$ is $n$. The zeta function of $K$ is 
\begin{equation}
\label{eq:zetadefsum}
\zeta_K(s)=\sum_{n=0}^\infty\frac{b_n(K)}{q^{ns}}.
\end{equation}
Moreover, we can write $\zeta_K(s)$ using a different formulation. For a positive integer $d$, suppose $a_d(K)$ is the number of places of $K$ with degree $d$. Also, for a positive integer $m$, suppose $N_m(K)=\sum_{d|m} da_d(K)$. We have that
\begin{equation}
\label{eq:zetadefprod}
\zeta_K(s)=\prod_{d= 1}^\infty\left(1-\frac{1}{q^{ds}}\right)^{-a_d(K)}=\exp\left(\sum_{m=1}^\infty\frac{N_m(K)}{mq^{ms}}\right).
\end{equation}
In this section, we study $N_m(K)$ for positive integers $m$.

From \cite{NTFunc}*{Chapter 5}, we also have that $\zeta_K(s)$ converges absolutely as the sum $\zeta_K(s)=\sum_{n=0}^\infty\frac{b_n(K)}{q^{ns}}$ and product $\zeta_K(s)=\prod_{d= 1}^\infty\left(1-\frac{1}{q^{ds}}\right)^{-a_d(K)}$ for $s\in\mathbb{C}$ such that $\text{Re}(s)>1$. In particular, if $s$ is a real number that is greater than one, then $\zeta_K(s)$ converges to a real number that is greater than one.

There has been a significant amount of research about global function fields with many rational places, which are places with degree $1$. Examples of papers about this topic are \cite{AuerRCF}, \cite{RP1}, and \cite{RP2}. On the other hand, in this paper we construct global function fields that do not have places of certain degrees. Particularly, the global function fields we construct do not have any rational places. We state the main result of this section.

\begin{theorem}
\label{thm:closeto1}
Suppose $s$ and $\epsilon$ are real numbers such that $s>1$ and $\epsilon>0$. There exists a global function field $K$ such that the full constant field of $K$ is $\mathbb{F}_q$ and $\zeta_K(s)<1+\epsilon$.
\end{theorem}

We prove \Cref{thm:closeto1} in \Cref{subsec:mainproof}. Note that the constructions in \Cref{subsec:projalg} are used to prove \Cref{thm:closeto1}.

\subsection{Projective algebraic curves}
\label{subsec:projalg}

\begin{definition}
\label{def:projalg}

For positive integers $n$ and $r$ such that $r>1$, let $C_{n,r}$ be the projective algebraic curve over $\mathbb{P}^2(\mathbb{F}_q)$ with equation
\begin{align*}
C_{n,r}:\, & x^{rq^n}-x^ry^{rq^n-r}+y^{rq^n} -x^rz^{rq^n-r}+x^ry^{q^n-1}z^{(r-1)q^n-r+1}-y^rz^{rq^n-r}
+z^{rq^n}=0.
\end{align*}
\end{definition}

\begin{theorem}
\label{thm:smooth}
Suppose $n$ and $r$ are positive integers. Assume that $r>1$ and $r\equiv 1\pmod{p}$. If $p\geq 3$, assume that $r$ divides $q^n-1$. Then, $C_{n,r}$ is smooth as a curve over $\mathbb{P}^2(\overline{\mathbb{F}_q})$.
\end{theorem}

\begin{proof}
For the sake of contradiction, let $(\alpha,\beta,\gamma)$ be a singular point on $C_{n,r}$ in $\mathbb{P}^2(\overline{\mathbb{F}_q})$. In order to prove this theorem, we use the fact that the characteristic of $\overline{\mathbb{F}_q}$ is $p$.

Suppose $\gamma=0$. Because we cannot have $\alpha=\beta=0$, we must have that $\alpha\not=0$ and $\beta\not=0$. Afterwards, the partial derivative of $C_{n,r}$ with respect to $x$ at $(\alpha,\beta,0)$ is nonzero, which is a contradiction.

Next, suppose $\gamma\not=0$. Set $\gamma=1$. Because the partial derivative of $C_{n,r}$ with respect to $x$ at $(\alpha,\beta, 1)$ is $0$, we get that
\begin{equation}
\label{eq:partialx}
\alpha^{r-1}(\beta^{rq^n-r}-\beta^{q^n-1}+1)=0.
\end{equation}
Since the partial derivative of $C_{n,r}$ with respect to $y$ at $(\alpha,\beta,1)$ is $0$,
\begin{equation}
\label{eq:partialy}
\alpha^r\beta^{rq^n-r-1}-\alpha^r\beta^{q^n-2}-\beta^{r-1}=0.
\end{equation}
If $\alpha=0$, then $\beta=0$ from \pref{eq:partialy}, but this is a contradiction because $(0,0,1)$ is not a point on $C_{n,r}$. Therefore, $\alpha\not=0$. Because of this, \pref{eq:partialx} implies that $\beta\not=0$ as well. Note that because $\alpha$ and $\beta$ are nonzero, the partial derivative of $C_{n,r}$ with respect to $z$ at $(\alpha,\beta,1)$ is $0$. Therefore, $\alpha^r+\beta^r=0$. Also, because $(\alpha,\beta,1)$ is a point on $C_{n,r}$,
\[
\alpha^{rq^n}+\beta^{rq^n}-\beta^r+1=0.
\]
However, $\alpha^{rq^n}+\beta^{rq^n}=(\alpha^r+\beta^r)^{q^n}=0$, giving that $\beta^r=1$ and $\alpha^r=-\beta^r=-1$. Because $\alpha\not=0$, \pref{eq:partialx} gives that
\[
\beta^{q^n-1}=\beta^{rq^n-r}+1=2.
\]
If $p=2$, this is a contradiction to $\beta\not=0$. Suppose $p\geq 3$. Then, $\beta^{q^n-1}=1$ since $r$ divides $q^n-1$, which is a contradiction to $\beta^{q^n-1}=2$. We are done.
\end{proof}

Suppose $n$ and $r$ are positive integers that satisfy the conditions of \Cref{thm:smooth}. The theorem implies that $C_{n,r}$ is smooth as a curve over $\mathbb{P}^2(\overline{\mathbb{F}_q})$ so $C_{n,r}$ is absolutely irreducible. From \cite{AuerRCF}, we therefore have that $\mathbb{F}_q(C_{n,r})$ is a global function field. Using \cite{AuerRCF}, we also have that for a positive integer $m$, $N_m(\mathbb{F}_q(C_{n,r}))$ is the number of points on $C_{n,r}$ as a projective curve over $\mathbb{P}^2(\mathbb{F}_{q^m})$. 

\subsection{The number of points on \texorpdfstring{$C_{n,r}$}{}}
\label{subsec:numpoints} 

Suppose $m$, $n$, and $r$ are positive integers such that $r>1$. We study the number of points on $C_{n,r}$ as a projective curve over $\mathbb{P}^2(\mathbb{F}_{q^m})$ in this subsection. Note that $n$ and $r$ do not necessarily follow the conditions of \Cref{thm:smooth}. The results of this subsection will be used to prove \Cref{thm:closeto1} in \Cref{subsec:mainproof}.

\begin{proposition}
\label{prop:numpoints1}
Suppose $m$, $n$, and $r$ are positive integers such that $r>1$ and $m$ divides $n$. The number of points on $C_{n,r}$ as a projective curve over $\mathbb{P}^2(\mathbb{F}_{q^m})$ is $0$.
\end{proposition}

\begin{proof}
For the sake of contradiction, assume $(\alpha,\beta,\gamma)\in \mathbb{P}^2(\mathbb{F}_{q^m})$ is a point on $C_{n,r}$. Suppose $\gamma=0$. Note that at least one of $\alpha$ and $\beta$ must be nonzero. Therefore, we must have that both $\alpha$ and $\beta$ are nonzero. With this, we can let $\beta=1$. Afterwards, we get that $\alpha^{rq^n}-\alpha^r+1=0$, a contradiction. Next, suppose $\gamma\not=0$; particularly, we can let $\gamma=1$. If $\beta=0$, we get $\alpha^{rq^n}-\alpha^r+1=0$, a contradiction. Assume that $\beta\not=0$. Then, using $\beta^{q^n-1}=1$, we also get $\alpha^{rq^n}-\alpha^r+1=0$, a contradiction. We are done.
\end{proof}

For a nonzero polynomial $P(x)$ in $\overline{\mathbb{F}_p}[x]$, define $v(P(x))$ to be the largest integer $e$ such that $x^e$ divides $P(x)$. Suppose $m$ is a positive integer. Because $\mathbb{F}_{q^m}\subset\overline{\mathbb{F}_p}$, we can consider $v$ to be a function over $\mathbb{F}_{q^m}[x]$.

\begin{lemma}
\label{lemma:polynumpoints}
Let $m$ and $d$ be positive integers such that $d$ is a proper divisor of $m$. Let $P(x)$ be a nonzero polynomial in $\mathbb{F}_{q^m}[x]$ such that $v(P(x))<\text{deg}(P(x))$. Suppose $A$ is the set of $w\in\mathbb{F}_{q^m}$ such that $w^{\frac{q^m-1}{q^d-1}}=1$. The number of $w\in A$ such that $P(w)\in A$ is at most 
\[
\frac{2(\deg(P(x))-v(P(x)))(q^{m-d}-1)}{q^d-1}.
\]
\end{lemma}

\begin{proof}
If $w\in A$,
\[
    P(w)^{\frac{q^m-1}{q^d-1}} =\prod_{i=0}^{\frac{m}{d}-1}P\left(w^{q^{di}}\right)
    = \prod_{i=0}^{\frac{m}{d}-2} P(w^{q^{di}}) P\left(w^{-\frac{q^{m-d}-1}{q^d-1}}\right).
\]
Suppose $Q(x)\in\mathbb{F}_{q^m}[x]$ is
\begin{align*}
Q(x) = & \prod_{i=0}^{\frac{m}{d}-2} P\left(x^{q^{di}}\right)x^{\frac{\text{deg}(P)(q^{m-d}-1)}{q^d-1}}P\left(x^{-\frac{q^{m-d}-1}{q^d-1}}\right)  -x^{\frac{\text{deg}(P)(q^{m-d}-1)}{q^d-1}}.
\end{align*}

We have that if $w\in A$ and $P(w)\in A$, then $w\not=0$ and $Q(w)=0$. Therefore, the number of $w\in A$ such that $P(w)\in A$ is at most the number of elements of $\mathbb{F}_{q^m}^{\times}$ that are roots of $Q(x)$. Note that $Q(x)\not=0$. Moreover,
\[
v(Q(x))=\frac{v(P(x))(q^{m-d}-1)}{q^d-1}
\]
and 
\[
\deg(Q(x))=\frac{(2\deg(P(x))-v(P(x)))(q^{m-d}-1)}{q^d-1}.
\]
With this, the number of elements of $\mathbb{F}_{q^m}^{\times}$ that are roots of $Q(x)$ is at most
\[
\deg(Q(x))-v(Q(x))=\frac{2(\deg(P(x))-v(P(x)))(q^{m-d}-1)}{q^d-1}.
\]
This completes the proof.
\end{proof}

\begin{theorem}
\label{thm:numpoints2}
Suppose $m$, $n$, and $r$ are positive integers such that $r>1$ and $m$ does not divide $n$. Let $d=\gcd(m,n)$. The number of points on $C_{n,r}$ as a projective curve over $\mathbb{P}^2(\mathbb{F}_{q^m})$ is at most
\[
\min(q^m-1, rq^{m-d})+r(q^d-1)\min(1+2r(q^{m-d}-1), q^m)+rq^m.
\]
\end{theorem}

\begin{proof}
Assume $(\alpha,\beta,0)$ is a point on $C_{n,r}$ in $\mathbb{P}^2(\mathbb{F}_{q^m})$. Then, $\alpha^{rq^n}-\alpha^r\beta^{rq^n-r}+\beta^{rq^n}=0$. Because we cannot have $\alpha=\beta=0$, $\alpha\not=0$ and $\beta\not=0$. Therefore, we can let $\beta=1$. The number of points for this case is then the number of solutions to $\alpha^{rq^n}-\alpha^r+1=0$ for $\alpha\in\mathbb{F}_{q^m}$. Suppose $l$ is the positive integer such that $ld$ is the remainder when $n$ is divided by $m$. Because $\alpha\not=0$, $\alpha^{q^m-1}=1$, which means that $\alpha^{q^n}=\alpha^{q^{ld}}$. Using $\alpha^{rq^n}-\alpha^r+1=0$, we have that $\alpha^{rq^{ld}}-\alpha^r+1=0$. Since $ld\leq m-d$, the number of points for this case is at most $\min(q^m-1, rq^{m-d})$.

Next, we consider points $(\alpha,\beta,1)$ on $C_{n,r}$ in $\mathbb{P}^2(\mathbb{F}_{q^m})$. Let
\[
f(x, y)=x^{q^n}-(y^{rq^n-r}-y^{q^n-1}+1)x+y^{rq^n}-y^r+1=0
\]
and 
\[
C(y)=y^{rq^n-r}-y^{q^n-1}+1.
\]
Note that $(\alpha,\beta,1)$ is on $C_{n,r}$ if and only if $f(\alpha^r,\beta)=0$.

Suppose that $S_1$ is the set of $\beta\in\mathbb{F}_{q^m}$ such that there are no solutions to $f(\alpha^r,\beta)=0$ for $\alpha\in\mathbb{F}_{q^m}$. Afterwards, we consider $\beta$ such that $\beta\notin S_1$.

Suppose $S_2$ is the set of $\beta\in\mathbb{F}_{q^m}\backslash S_1$ such that $C(\beta)=0$. Suppose $\beta\in S_2$. The number of solutions to $f(\alpha',\beta)=0$ for $\alpha'\in\mathbb{F}_{q^m}$ is at most $1$ (actually equal to $1$ because $\beta\notin S_1$), so the number of solutions to $f(\alpha^r, \beta)=0$ for $\alpha\in\mathbb{F}_{q^m}$ is at most $r$. This gives at most $r|S_2|$ points for the case $\beta\in S_2$.

For the next step, suppose $S_3$ is the set of $\beta\in\mathbb{F}_{q^m}\backslash (S_1\cup S_2)$ such that there are no solutions to $u^{q^n-1}=C(\beta)$ for $u\in \mathbb{F}_{q^m}$. Suppose $\beta\in S_3$. Also, suppose that for $w\in\mathbb{F}_{q^m}$, $f(w,\beta)=0$; because $\beta\notin S_1$, $w$ exists. For $l\in\mathbb{F}_{q^m}$, $f(w+l,\beta)=0$ if and only if $l^{q^n}-C(\beta)l=0$. However, because $\beta\in S_3$, this is true for only $l=0$. With this, $w$ is the only element of $\mathbb{F}_{q^m}$ that is a root of $f(x,\beta)$. Also, $\alpha^r=w$ for $\alpha\in\mathbb{F}_{q^m}$ has at most $r$ solutions. Then, the number of solutions to $f(\alpha^r,\beta)=0$ for $\alpha\in\mathbb{F}_{q^m}$ is at most $r$. Therefore, we have that the number of points from the case $\beta\in S_3$ is at most $r|S_3|$.

Suppose $S_4=\mathbb{F}_{q^m}\backslash (S_1\cup S_2\cup S_3)$ and $\beta\in S_4$. Suppose that for $w\in\mathbb{F}_{q^m}$, $f(w,\beta)=0$. Similarly to the case for $S_3$, for $l\in\mathbb{F}_{q^m}$, $f(w+l,\beta)=0$ if and only if $l^{q^n}-C(\beta)l=0$. A solution to this is $l=0$. Suppose $l_0\in\mathbb{F}_{q^m}^{\times}$ and $l_0^{q^n-1}=C(\beta)$. Note that because $\beta\in S_4$, $l_0$ exists. For $a\in\mathbb{F}_{q^m}^{\times}$, $(al_0)^{q^n-1}=C(\beta)$ if and only if $a^{q^n-1}=1$. But, because $a^{q^m-1}=1$, $a^{q^n-1}=1$ if and only if $a^{q^d-1}=1$. However, $a^{q^d-1}=1$ has $q^d-1$ solutions for $a\in\mathbb{F}_{q^m}^{\times}$. Therefore, $l^{q^n}-C(\beta)l=0$ has $q^d-1$ solutions for $l\in\mathbb{F}_{q^m}^{\times}$. We then have that $l^{q^n}-C(\beta)l=0$ has $q^d$ solutions for $l\in\mathbb{F}_{q^m}$. Following this, $f(x,\beta)=0$ has $q^d$ solutions for $x\in\mathbb{F}_{q^m}$. Therefore, $f(\alpha^r,\beta)=0$ has at most $rq^d$ solutions for $\alpha\in\mathbb{F}_{q^m}$. Afterwards, the number of points from the case $\beta\in S_4$ is at most $rq^d|S_4|$.

We upper bound $|S_4|$. Let $T$ be the set of $w\in\mathbb{F}_{q^m}$ such that $w^{\frac{q^m-1}{q^d-1}}=1$ and
$(w^r-w+1)^{\frac{q^m-1}{q^d-1}}=1$. From \Cref{lemma:polynumpoints}, $|T|\leq\frac{2r(q^{m-d}-1)}{q^d-1}$. Suppose $\beta\in S_4$ and $\beta\not=0$. We have that $\beta^{\frac{(q^n-1)(q^m-1)}{q^d-1}}=1$. Also, because $C(\beta)\not=0$ and $u\in\mathbb{F}_{q^m}$ exists such that $u^{q^n-1}=C(\beta)$, $C(\beta)^{\frac{q^m-1}{q^d-1}}=1$. Therefore, $\beta^{q^n-1}\in T$. Suppose $w\in T$. Note that the number of solutions to $\beta^{q^n-1}=w$ for $\beta\in\mathbb{F}_{q^m}$ is at most $q^d-1$. After accounting for the case that $0$ could be an element of $S_4$, we obtain that
\[
|S_4|\leq 1+(q^d-1)|T|\leq 1+2r(q^{m-d}-1)
\]
and therefore that $|S_4|\leq\min(1+2r(q^{m-d}-1), q^m)$.

We then have that the number of points in $\mathbb{P}^2(\mathbb{F}_{q^m})$ on $C_{n,r}$ is at most
\begin{align*}
    & \min(q^m-1, rq^{m-d})+r|S_2|+r|S_3|+rq^d|S_4|  \\ 
    & \leq \min(q^m-1, rq^{m-d}) +r(q^m-|S_4|)+rq^d|S_4| \\
    & \leq \min(q^m-1, rq^{m-d})+r(q^d-1)\min(1+2r(q^{m-d}-1), q^m)+rq^m.
\end{align*}
This finishes the proof.
\end{proof}

\subsection{Proof of 
\texorpdfstring{\Cref{thm:closeto1}}{}}
\label{subsec:mainproof}

Suppose $N$ is an integer such that $N\geq 2$. Let $K_N=\mathbb{F}_q(C_{N!,\,p+1})$. Note that because $n=N!$ and $r=p+1$ satisfy the conditions of \Cref{thm:closeto1}, $K_N$ is a global function field and the full constant field of $K_N$ is $\mathbb{F}_q$. From \Cref{prop:numpoints1}, we have that $N_m(K_N)=0$ for positive integers $m$ such that $m$ divides $N!$. Moreover, from \Cref{thm:numpoints2}, $N_m(K_N)\leq (2p^2+5p+3)q^m$ for positive integers $m$ such that $m$ does not divide $N!$. After using (\ref{eq:zetadefprod}), 
\[
\log(\zeta_{K_N}(s))=\sum_{m=1}^\infty \frac{N_m(K_N)}{mq^{ms}}.
\]
However,
\begin{align*}
\sum_{m=1}^\infty\frac{N_m(K_N)}{mq^{ms}}
& \leq (2p^2+5p+3)\sum_{m= N+1}^\infty\frac{1}{mq^{m(s-1)}} \\
& \leq \frac{2p^2+5p+3}{N+1}\sum_{m=N+1}^\infty \frac{1}{q^{m(s-1)}} \\
& = \frac{2p^2+5p+3}{N+1}\cdot\frac{1}{q^{(N+1)(s-1)}}\cdot\frac{1}{1-\frac{1}{q^{s-1}}}.
\end{align*}
Therefore, $\lim_{N\rightarrow\infty}\log(\zeta_{K_N}(s))=0$. It follows that $\zeta_{K_N}(s)<1+\epsilon$ for sufficiently large $N$, completing the proof.

\bibliography{mybib.bib}

\end{document}